\documentclass[]{amsart}       % onecolumn (second format)
\usepackage{mathptmx}      % use Times fonts if available on your TeX system
\usepackage{amsmath}
\usepackage{amssymb}
\usepackage{amsthm}
\newcommand{\SLZ}{\mathrm{SL}_{2}(\ZZ)}
\newcommand{\MPZ}{\mathrm{Mp}_{2}(\ZZ)}
\renewcommand{\H}{\mathbb{H}}
\newcommand{\ZZ}{\mathbb{Z}}

\newcommand{\QQ}{\mathbb{Q}}

\newcommand{\Arg}{\textup{Arg}}
\newcommand{\ord}{ \textup{ord}}

\newcommand{\fqmQ}{\ensuremath{\mathcal{Q}}}
\newcommand{\eqd}[1]{\varepsilon_{\fqmQ,#1}}
\newcommand{\eq}{\varepsilon_{\fqmQ} }
\newcommand{\si}{{\sigma}_{w}(\fqmQ)}

\newcommand{\Scheps}[1]{\lambda_{#1}}
\renewcommand{\r}{\ensuremath{{\rho}_{\fqmQ}}}
\newcommand{\rt}{\ensuremath{{ \widetilde{\rho} }_{\fqmQ}}}
\newcommand{\ch}{\ensuremath{{\chi}_{\fqmQ}}}
\newcommand{\sign}{\textup{sign}}
\newcommand{\oddity}{\textup{oddity}}
\newcommand{\pexcess}{p\textup{-excess}}
\newcommand{\lc}{\ensuremath{l_c}}

\newcommand{\gammafak}{g} %\mathfrak{g}}
\newcommand{\mat}[1]{{\bf #1}}

\numberwithin{equation}{section}
\theoremstyle{plain}
\newtheorem{theorem}{Theorem}[section]
\newtheorem{corollary}[theorem]{Corollary}
\newtheorem{theorem*}{Theorem}
\newtheorem{lemma}[theorem]{Lemma}
\theoremstyle{definition}
\newtheorem{example}[theorem]{Example}
\newtheorem{assumption}[theorem]{Assumption}

\newtheorem{definition}[theorem]{Definition}
\theoremstyle{remark}
\newtheorem{remark}[theorem]{Remark}
\theoremstyle{plain}
\newtheorem*{acknowledgements}{Acknowledgements}
\makeatother

\begin{document}

\title{Weil Representations associated to finite quadratic modules}

\author{Fredrik Str{\"o}mberg}

\address{Fredrik Str{\"o}mberg \newline
\indent Fachbereich Mathematik, Technische Universit{\"a}t Darmstadt, Schlossgartenstrasse 7 \newline
\indent  64289 Darmstadt, Germany}
\email{stroemberg@mathematik.tu-darmstadt.de}           %  \\


\begin{abstract}
To a finite quadratic module, that is, a finite abelian group $D$
together with a non-singular quadratic form $Q:D\rightarrow\QQ/\ZZ$,
it is possible to associate a representation of either the modular
group, $\SLZ,$ or its metaplectic cover, $\MPZ$, on $\mathbb{C}\left[D\right]$,
the group algebra of $D$. This representation is usually called the
Weil representation associated to the finite quadratic module. The
main result of this paper is a general explicit formula for the matrix
coefficients of this representation. The formula, which involves the
$p$-adic invariants of the quadratic module, is given in a way which
is easy to implement on a computer. The result presented completes
an earlier result by Scheithauer for the Weil representation associated to a discriminant form of even
signature.
\end{abstract}

\maketitle

\section{Introduction}

\subsection{Historical background}

The theory of theta functions originated with letters from Euler to Goldbach
in the years 1748--1750 \cite[Letters 115-133]{MR0225627} where the
now classical theta function, which for  $\tau\in\H=\left\{ x+iy\,|\, y>0\right\}$ is defined by 
\[
\vartheta\left(\tau\right)=\sum_{n=-\infty}^{\infty}e^{\pi i\tau n^{2}},
\]
was introduced by Euler as a means to study the decomposition of integers into sums
of squares. 
A comprehensive history of this and other theta functions
is given by Krazer \cite{0212.42901} (see also \cite{MR1711085}).
There is an intimate connection between the subsequent development of the theory of theta functions
and the type of Weil representation we consider in this paper. 
%The history of the type of Weil representations we consider in this
%paper is intimately connected to the theory of theta functions, whose
%origin can be traced back at least to letters from Euler to Goldbach
%in the years 1748--1750 \cite[Letters 115-133]{MR0225627} where the
%now classical theta function, which for  $\tau\in\H=\left\{ x+iy\,|\, y>0\right\}$ is defined by 
%\[
%\vartheta\left(\tau\right)=\sum_{n=-\infty}^{\infty}e^{\pi i\tau n^{2}},
%\]
%was introduced by
%Euler as a means to study the decomposition of integers into sums
%of squares. 
 For our purposes, the
next important step was taken by Poisson \cite[p.\ 420]{poisson_summation}
(cf.~\cite[p.\ 260]{MR0260557}) who proved the transformation law
\[
\vartheta\left(-\tau^{-1}\right)=\sqrt{-i\tau}\vartheta\left(\tau\right)
\]
using Fourier-theoretic methods, essentially what we today refer to
as the Poisson summation formula. If we set $\theta(\tau)=\vartheta(2\tau)$
it is known (cf.~e.g.~\cite[p.\ 46]{Iwaniec:topics}) that 
%that for $A\in\Gamma_{0}(4)$ 
\[
\theta(\mat{A}\tau)=v_{\theta}(\mat{A})\sqrt{c\tau+d}\,\theta(\tau),\quad \textrm{for } \mat{A}\in\Gamma_0(4), 
\]
where $v_{\theta}(\mat{A})=\left(\frac{c}{d}\right)$ if $d\equiv1$ (mod $4$)
and $-i\left(\frac{c}{d}\right)$ if $d\equiv3$ (mod $4$) 
(cf.~Remark \ref{rem:chi_as_v_theta}). Jacobi \cite{MR0260557} considered theta
functions in connection with elliptic integrals and was therefore
naturally led to define theta functions in two variables, e.g. $\vartheta\left(\tau\right)=\vartheta_{3}\left(\tau,0\right)$
where $\vartheta_{3}(\tau,z)=\sum_{n\in\ZZ} e^{\pi i\tau n^{2}+2\pi izn}$
(cf.~e.g.~\cite[p.~501]{MR0260557}) is an example of what is now called
a \emph{Jacobi form}. An introduction to the general theory of classical Jacobi forms
is given by e.g. Eichler and Zagier \cite{eichler-zagier}; note that the definition of Jacobi
form here does not include $\vartheta$. 
%A discussion of $\vartheta$ as a Jacobi form of index $\frac{1}{2}$ is given 
The function $\vartheta$ was introduced as a Jacobi form of index $\frac{1}{2}$ 
by Skoruppa \cite{MR806354} and Gritsenko \cite{MR982481}. A representation-theoretical 
approach to Jacobi forms is given by Berndt and Schmidt \cite{MR1634977}.
Nowadays, a theta function is usually considered in association with a lattice. 
From this point of view it has a natural interpretation as a vector-valued
modular form for the Weil representation corresponding to this lattice. An introduction to the modern theory of 
theta functions is given by e.g.~Koecher and Krieg~\cite{MR1711085} and Ebeling~\cite{MR1280458}. 
The key steps in the historical development of this theory were taken by e.g.~Hermite \cite{hermite:1858},
Hecke \cite{MR1512360}, Schoeneberg \cite{MR1513241}, Kloosterman
\cite{kloosterman}, Pfetzer \cite{MR0059945}, Weil \cite{MR0165033},
Wolfart and Nobs \cite{MR0429742} and Wolfart \cite{MR0429743}. 

In this setting the function $\theta$ can be viewed as one component
of a vector-valued theta function associated to the lattice $\ZZ$
together with the quadratic form $x\mapsto x^{2}$. 

Through the connection with theta series it is possible to obtain
relationships between vector-valued modular forms for the Weil representation
and other types of automorphic forms. The most direct
relationship is the identification between Jacobi forms and vector-valued
modular forms for the Weil representations associated to the index of
the Jacobi form. Cf. e.g.~\cite{eichler-zagier,MR806354,MR2512363}.
For example, the space of classical holomorphic Jacobi forms, $J_{k,m}^{+}$, of weight $k$ and positive integer index $m$,
corresponds to vector-valued modular forms for the
dual Weil representation associated to the lattice $\ZZ$
together with the quadratic form $x\mapsto mx^{2}$ 
(see Example \ref{exa:Nx^2}).

Relationships to modular forms on orthogonal groups and automorphic
products are given by e.g.~Borcherds \cite{Bo,MR1773561}, Bruinier \cite{MR1903920,Br2}
and Scheithauer \cite{scheithauer:weil_rep,MR2221135}. In a representation-theoretical setting, 
Gelbart \cite{MR0424695} used Weil representations  
to describe and decompose automorphic representations of metaplectic
(adele) groups, as well as describe correspondences between half-integral
and integral weight automorphic forms. Additional information in this context is given by Niwa \cite{niwa:half_integral}
and Shintani \cite{shintani:75:onconstruction}.

\subsection{Statement of the main result}

Let $D$ be a finite abelian group and $Q:D\rightarrow\QQ/\ZZ$
a quadratic form with non\-de\-ge\-ne\-rate associated bi-linear form $B\left(x,y\right):=Q\left(x+y\right)-Q\left(x\right)-Q\left(y\right)$.
That is, $Q\left(ax\right)=a^{2}Q\left(x\right)$ for all $a\in \ZZ$  
and $x\in D$ and the map $x\mapsto B\left(x,\cdot\right)$ is a linear
isomorphism between $D$ and $\text{Hom}\left(D,\QQ/\ZZ\right)$
for all $x\in D$. The pair $\fqmQ=\left(D,Q\right)$ is said
to be a \emph{finite quadratic module} (FQM). The \emph{level}
of $\fqmQ$ is defined as the smallest natural number, $l$,
such that $lQ\left(x\right)\in\ZZ$ for all $x\in D$. It turns
out that $\fqmQ$ determines a representation
of either the modular group, $\SLZ$, or its metaplectic cover, $\MPZ$,
 on $\mathbb{C}\left[D\right]$, the group algebra of $D$. Which one
of these two cases occur depends on the so-called signature of $\fqmQ$,
cf.~\eqref{eq:oddity_formula}. This representation can be
viewed as a special case of a construction carried out by Weil \cite{MR0165033}
and is therefore usually called the \emph{Weil representation} associated
to $\fqmQ$. For more information on the general theory of Weil
representations associated to finite quadratic modules see Skoruppa
\cite{skoruppa-weilrep}. 

The canonical example of a finite quadratic module is the so-called
\emph{discriminant form} associated to an even lattice $L$, with
non-degenerate bilinear form, given by the discriminant group $D=L'/L$,
where $L'$ is the dual lattice of $L$, together with the reduction
modulo one of the quadratic form on $L'$. Note that Borcherds \cite{MR1773561}
uses the term discriminant form for any FQM.
\begin{example}
\label{exa:Nx^2}Let $N$ be a positive integer and $L$ be the lattice
$\ZZ$ with quadratic form $q:x\mapsto Nx^{2}$. Then $D=L'/L\simeq\frac{1}{2N}\ZZ/\ZZ$
together with the quadratic form $Q\left(x\right)=Nx^{2}$ (mod $1$) is
an example of an FQM with level $4N$ and signature $1$. 
\end{example}
The main purpose of this paper is to obtain an explicit simple
formula for the action of $\SLZ$ or $\MPZ$ on the Weil representation corresponding to an arbitrary finite quadratic module $\fqmQ$.
Our main result is the following. For the precise statement see Theorem \ref{thm:main_thm} and Remark \ref{rem:main_thm_mat_elt}.

\begin{theorem*}
Let $\fqmQ=(D,Q)$ be an FQM and $\alpha,\beta\in D$. 
If $\mat{A}=\left(\begin{smallmatrix}a & b\\ c & d\end{smallmatrix}\right)\in\SLZ$ 
then the matrix coefficient $\r(\mat{A})_{\alpha\beta}$ is
given by
\[
\r(\mat{A})_{\alpha\beta}=\xi\left(\mat{A}\right)\sqrt{\left|D_{c}\right|/\left|D\right|}
e\big(acQ\left(\alpha'\right)+aB\left(x_{c},\alpha'\right)-bdQ\left(\beta\right)+bB\left(\beta,\alpha\right)\big)
\]
if there is an element $\alpha'\in D$ such that $\alpha=d\beta+x_{c}+c\alpha'$
and otherwise $\r\left(\mat{A}\right)_{\alpha,\beta}=0$. Here $\xi\left(\mat{A}\right)$ is an explicit eight-root of unity, given
in terms of $p$-adic invariants of $\fqmQ$.
The element $x_{c}\in D$ is of order $2$ and given by the $2$-components of $\fqmQ$, and $D_{c}$ consists
of all elements in $D$ of order dividing $c$. 
\end{theorem*}

In this context ``simple'' means that the formula for $\xi$ does
not involve any summation or integration, only elementary arithmetic
functions, e.g. Kronecker symbols. It is not difficult to obtain a
formula for $\r$ as a projective representation, that is, to obtain
the theorem above with an unknown factor $\xi$ of absolute value
one given in the form of a Gauss sum; together with a cocycle in
the case of odd signature. Evaluating this Gauss sum explicitly is,
however, a totally different matter. If $\fqmQ$ is a discriminant
form of level $l$ then explicit (and simple) formulas for $\xi$
has been known for a long time in the case when $\mat{A}$ belongs to certain
congruence subgroups of level $l$. Cf.~e.g.~Schoeneberg \cite[p.~518]{MR1513241},
Pfetzer \cite[pp.~451-452]{MR0059945} and Kloosterman \cite[I.\S4]{kloosterman}.
See also e.g.~\cite{skoruppa-weilrep}, \cite[Lemma 3.2]{MR1773561}
as well as Lemma \ref{lem:action_on_gamma_0^0} and Lemma \ref{lem:action_on_gamma_0}.
%below. 

In contrast to these simple formulas the earlier formulas for the full modular
or metaplectic group all involved certain sums of Gauss-type with a length depending on
the particular element of the group. Cf.~e.g.~\cite[p.\ 516]{MR1513241},
\cite[p.\ 450]{MR0059945}, \cite[I. Thm.\ 1]{kloosterman}, \cite[Prop.\ 1.6]{shintani:75:onconstruction},
\cite[p.\ 519]{MR909238} and \cite[Prop.\ 3.2]{MR1280458}. 
%One of the main points of the current paper is that we evaluate all such
%sums explicitly.

This situation improved a great deal when Scheithauer \cite{scheithauer:weil_rep}
 obtained an explicit formula for the root of unity $\xi(\mat{A})$,
for any $\mat{A}$ in the modular group, in terms of the $p$-adic
invariants of the underlying lattice. Although these results were
restricted to discriminant forms of even signature, the generalization
to arbitrary FQMs of even signature is more or less immediate.
However, to obtain the corresponding results for FQMs of odd signature
requires more effort, mainly due to the necessity to
work with the metaplectic group.

The main points of the current paper are that we obtain explicit formulas for the Weil representation associated to 
\emph{any} FQM, without restrictions on the signature, and that all Gauss-type sums are explicitly
evaluated in terms of $p$-adic invariants. 

The computational aspects were foremost in mind when we obtained these formulas; we needed efficient
algorithms for the Weil representation in order to compute vector-valued
Poincar{\'e} series \cite{lfg} and harmonic weak Maass forms \cite{compmaass}. 
The formula stated in the Main Theorem is implemented as part of a package \cite{fqm} written in Sage \cite{sage} for computing
with finite quadratic modules.  
%The motivation for obtaining these explicit formulas is mainly computational 
%for the Weil representation associated to FQMs of odd signature was mainly computational. 
%More precisely, it was necessary to obtain 

\subsection{Notational conventions}

To simplify the exposition we write $e\left(x\right)=e^{2\pi ix}$, $e_{r}(x)=e(\frac{x}{r})$ and use $\left(a,b\right)=\gcd\left(a,b\right)$.
Furthermore we always use the Kronecker extension of the Jacobi symbol,
$\left(\frac{c}{d}\right)$. For odd $c$ and $d$ with $d>0$ this
is the usual quadratic residue symbol, and for arbitrary integers $c,d$
we define $(\frac{c}{d})$ by complete multiplicativity,
using $(\frac{c}{d})=\sign(c)(\frac{c}{-d})$,
$(\frac{2}{d})=(\frac{d}{2})$ for odd $d$,
$(\frac{d}{0})=(\frac{0}{d})=1$ if $d=\pm1$ and $0$ otherwise. 
If $p$ is a prime number and $n$ is an integer we define the \emph{$p$-adic additive valuation} of $n$ by $\ord_p(n)=k$ if $p^k$ is the largest power of $p$ dividing $n$.
This is extended to the rational numbers by setting $\ord_p(\frac{c}{d})=\ord_p(c)-\ord_p(d)$,
and we use $|x|_{p}=p^{-\ord_{p}(x)}$ to denote the $p$-adic absolute value of $x$. 
For a finite set $S$ we use $\left|S\right|$
to denote the number of elements in $S$. If $a$ and $b$ are two integers then
the Hilbert symbol at infinity is defined as $\left(a,b\right)_{\infty}=-1$
if $a<0$ and $b<0$, and $\left(a,b\right)_{\infty}=1$ otherwise.
For a complex number, $z$, we use $\sqrt{z}$ to denote the principal
branch of the square root of $z$, that is, $\sqrt{z}=\sqrt{|z|}\exp\left(\frac{1}{2}i\Arg z\right)$
where $\Arg z\in\left(-\pi,\pi\right]$ is the principal branch of
the argument. Furthermore, let $\ZZ_{r}=\ZZ/r\ZZ$
and $\mathbb{\ZZ}_{r}^{\times n}=\ZZ_{r}\times\cdots\times\ZZ_{r}=$
($n$ times) and we write $m=\square$ to say that $m$ is the square
of an integer. 

For the remaining part of the paper we use the following convention:
If $\fqmQ$ is an FQM then the associated abelian group
is always denoted by $D$, the quadratic form by $Q$ and the associated
bilinear form by $B$. The structure of the paper is as follows: We
begin with a more detailed review of Jordan decompositions of finite
quadratic modules in Section \ref{sec:On-Jordan-Decompositions} and
follow this with an evaluation of those Gauss sums which are necessary
to express the local quantities of the number $\xi$ in the main theorem.
We then discuss the metaplectic cover of $\SLZ$ in more detail.
In Section \ref{sec:The-Weil-Representation} we then define the Weil
representation and show how it behaves under the action of certain
congruence subgroups of $\SLZ$, concluding with a precise formulation
of the main theorem. The proof of this is then presented in the last
section.

\section{On Jordan decompositions of finite quadratic modules}
\label{sec:On-Jordan-Decompositions}

The concept of a Jordan decomposition is well-known from linear algebra and it is also
very useful for FQMs. We repeat many facts from Scheithauer \cite{scheithauer:weil_rep} but aim to provide more details.
%In terms of notation related to the Jordan decomposition we stay close to that of Scheithauer \cite{scheithauer:weil_rep} and Borcherds \cite{MR1773561},
%but for the benefit of the reader we provide a few more details.
Let $\fqmQ=\left(D,Q\right)$ be an FQM and let $B$ be the associated bi-linear form.
An elementary, but important, observation is that the notion of a finite
quadratic module fits nicely into the general framework of abstract
lattices and quadratic forms over Dedekind domains, as defined by O'Meara
\cite[Part 4]{MR0347768}. %, who provides a thorough reference to the theory of quadratic forms over rings. 
Consider the action of $\ZZ$ on $D$ given by multiplication.
 For $c\in \ZZ$ we define  the map $\varphi_{c}:D\rightarrow D$ by $\varphi_c(\gamma)=c\gamma$ and then use $D_{c}$
and $D^{c}$ to denote the kernel and image of $\varphi_c$, respectively. 
That is, $D_{c}$ consists of all elements in $D$ of order dividing $c$, and $D^{c}$ is the set of all the $c$-th
powers of elements of $D$. Note that $D^{c}$ is the orthogonal complement of $D_{c}$. Define 
\[
D^{c*}=\left\{ \alpha\in D\,:\,\psi_{c,\gamma}\left(\alpha\right)=0 \text{ (mod $1$)}\,\,\forall\gamma\in D_{c}\right\} =\cap_{\gamma\in D_{c}}\text{Ker}\left(\psi_{c,\gamma}\right),
\]
where $\psi_{c,\gamma}\left(\alpha\right)=cQ\left(\gamma\right)+\left(\alpha,\gamma\right)$.
It is clear that if $\left(c,\left|D\right|\right)=\left(d,\left|D\right|\right)$
then $D_{c}=D_{d}$, $D^{c}=D^{d}$ and $D^{c*}=D^{d*}$. In particular,
if $\left(c,\left|D\right|\right)=1$ then $D_{c}=\left\{ 0\right\} $
and $D^{c}=D^{c*}=D$, and if $\left|D\right||c$ then $D^{c}=\left\{ 0\right\} ,$
$D_{c}=D$.  

It is well-known that $D$ can be written uniquely as a direct sum of
$p$-groups, that is, cyclic subgroups of prime-power orders. For
our purposes we need a refinement into $q$-groups: $D\left[q\right]=\left\{ \gamma\in D\,:\, q\gamma=0\right\} $
with $q=p^{k}$ for some prime $p$ dividing $\left|D\right|$ and
positive integer $k$. We introduce parameters $n\ge1$ (the rank of
$D\left[q\right]$) and $\epsilon\in\left\{ \pm1\right\}$ depending
on $Q$ restricted to $D\left[q\right]$. We write this decomposition
of $D$ as 
\begin{equation}
D=\bigoplus_{p\mid\,|D|}\bigoplus_{p|q}D\left[q^{n\epsilon}\right].\label{eq:Jordan-decomposition}
\end{equation}
When there is risk of confusion we write $n_{q}$ and $\epsilon_{q}$
instead of $n$ and $\epsilon$. This decomposition is orthogonal
with respect to $B$ and is called the \emph{Jordan decomposition}
of $\fqmQ$ (cf.~e.g.~\cite[\S 91C]{MR0347768}). It is unique
except for $p=2$ (cf.~e.g.~\cite[p.\ 381]{MR1662447}). When 
 $D$ is given we usually write $q^{\epsilon n}$ instead of
$D\left[q^{\epsilon n}\right]$. 

For later reference it is useful to have an explicit description of all
Jordan components which can appear. The following lemma is given by Scheithauer \cite[pp.\ 5-6]{scheithauer:weil_rep} 
and is also proven by Skoruppa \cite[Ch.~1]{skoruppa-weilrep}; it is not hard to deduce using results of  
Conway-Sloane \cite{MR1662447} and Nikulin \cite[\S 1]{MR525944} (see in particular \cite[p.~113]{MR525944} for the even $2$-adic components).
\begin{lemma}
\label{lem:def-of-component}The possible non-trivial Jordan components of $\fqmQ$
are the following:

$p>2$: $D\left[q^{\epsilon n}\right]\simeq\ZZ_{q}^{\times n}$
where $q=p^{k}$, $k\ge1$, $\epsilon=\pm1$ and $n\in\mathbb{N}$.
The indecomposable components are of the form: $D\left[q^{\epsilon}\right]\simeq\ZZ_{q}$,
generated by an element $\gamma$ of order $q$ with $Q\left(\gamma\right)\equiv\frac{a}{q}$ \textup{(mod $1$)}
for some $a\in\ZZ$ with $(\frac{2a}{p})=\epsilon$.
We define the \emph{$p$-excess} of this component by % (sometimes the $p$-signature is defined as the dimension $+\pexcess$) by
\[
\pexcess\left(q^{\epsilon n}\right)=n\left(q-1\right)+4k,\quad k=
\begin{cases}
1 & \text{if } q\ne\square \text{ and }\epsilon=-1,\\
0 & \text{otherwise.}
\end{cases}
\]

$p=2$, odd type: $D\left[q_{t}^{\epsilon n}\right]\simeq\ZZ_{q}^{\times n}$
where $q=2^{k},$ $k\ge1,$ $\epsilon=\pm1$, $n\in\mathbb{N}$, and
$t\in\ZZ/8\ZZ$ satisfies $t\equiv n$ \textup{(mod $2$)}. 
The indecomposable components are of the form: $D\left[q_{t}^{\epsilon}\right]\simeq\ZZ_{q}$
with $\epsilon=\left(\frac{t}{2}\right)$, generated by an element
$\gamma$ of order $q$ with $Q(\gamma)\equiv\frac{t}{2q}$ \textup{(mod $1$)}.
We define the \emph{oddity} of this component by
\[
\oddity\left(q_{t}^{\epsilon n}\right)=t+4k,\quad k=
\begin{cases}
1 & \text{if } q\ne\square\,\,\text{and }\epsilon=-1,\\
0 & \text{otherwise.}
\end{cases}
\]

$p=2$, even type: $D\left[q^{\epsilon2n}\right]\simeq\ZZ_{q}^{\times n}\times\ZZ_{q}^{\times n}$
where $q=2^{k},$ $k\ge1,$ $\epsilon=\pm1$, $n\in\mathbb{N}$. 
The indecomposable components are of the form
 $D\left[q^{2\epsilon}\right]\simeq\ZZ_{q}\times\ZZ_{q}$,
generated by two elements $\gamma$ and $\delta$ of order $q$. The gram matrix
of $B$ restricted to this component is 
\[
\frac{1}{q}\left(\begin{array}{cc}
0 & 1\\
1 & 0\end{array}\right) \textup{ (mod $1$)} \quad\text{if }\,\epsilon=1,\quad\frac{1}{q}\left(\begin{array}{cc}
2 & 1\\
1 & 2\end{array}\right)\textup{ (mod $1$)}\quad\text{if }\,\epsilon=-1,
\]
and we define the \emph{oddity } of this component by
\[
\oddity\left(q^{\epsilon n}\right)=4k,\quad k=
\begin{cases}
1 & \text{if }\,q\ne\square\,\text{ and }\,\epsilon=-1,\\
0 & \text{otherwise.}
\end{cases}
\]
\end{lemma}
\begin{remark}
The sum of two Jordan components with the same $q$ is given by multiplying the signs, adding the ranks and 
adding any subscripts $t$. Any trivial component, that is, if $q=1$ or $n=0$, have zero $\pexcess$ and oddity.
Sometimes the names ``type I'' and ``type II'', are used for the odd and even $2$-adic components, respectively.
\end{remark}
\begin{example}
Let $N$ be a positive integer and write $2N=p^{m_{p}}N_{p}$ with
$p\nmid N_{p}$ for each prime $p$ dividing $2N$. Then $\fqmQ_{N}=\left(D,Q\right)$
with $D\simeq\frac{1}{2N}\ZZ/\ZZ$ and $Q\left(x\right)=Nx^{2}$
has the following Jordan components. For $p>2$: $D\left[q^{\epsilon_{p}}\right]$
with $q=p^{m_{p}}$ and $\epsilon_{p}=(\frac{2N_{p}}{p})$.
For $p=2$: $D\left[q_{t}^{\epsilon_{2}}\right]$ with $q=2^{m_{2}}$,
$t=N_{2}$ and $\epsilon_{2}=(\frac{N_{2}}{2})$. For $p\ge2$ the $p$-adic
component is generated by $\gamma=\frac{N_{p}}{2N}$.
\end{example}
%\begin{definition}
If $J$ is a $p$-adic Jordan component we define the eight root of unity $\gammafak_{p}(J)$ by 
\begin{equation}
\gammafak_{p}(J)=
\begin{cases}
  e_{8}\big(-\pexcess(J)\big) & \text{if } p>2, \label{eq:def:gamma_p} \\
  e_{8}\big(\oddity(J)\big)& \text{if } p=2.
\end{cases}
\end{equation}
Using Lemmas \ref{lem:g_q=00003Dgamma_q_pgt2}, \ref{lem:g_q=00003Dgamma_q_odd_2}
and \ref{lem:g_q=00003Dgamma_q_even2} these factors can be computed explicitly and % we see %that
\begin{align*}
\gammafak_{p}\left(q^{n\epsilon}\right) & =\left(\frac{2a}{q}\right)e_{8}\big(n(1-q)\big)\quad \text{if }\, 2<p|q,\\
\gammafak_{2}\left(q_{t}^{n\epsilon}\right) & =\left(\frac{t}{q}\right)e_{8}(t)\quad\text{if }\, 2|q \textrm{ and }\,J\,\text{ is }\,\text{odd $2$-adic,}\\
\gammafak_{2}\left(q^{2n\epsilon}\right) & =\left(\frac{2-\epsilon}{q}\right)\quad \text{if }\, 2|q \text{ and }\, J\,\text{ is }\,\text{even $2$-adic}, 
\end{align*}
where $a$ is given in Lemma \ref{lem:def-of-component}. 
If $\fqmQ$ is an FQM then the oddity and $\pexcess$ of $\fqmQ$ are defined by summing over its Jordan components. 
Combining these local $p$-adic invariants we define the \emph{signature} of $\fqmQ$, $\sign(\fqmQ)\in\ZZ/8\ZZ$, by the \emph{oddity formula}:
\begin{equation}
\sign(\fqmQ)\equiv\oddity\left(\fqmQ\right)-\sum_{p>2}\pexcess(\fqmQ)\text{ (mod $8$)}.\label{eq:oddity_formula}
\end{equation}
The proof of \eqref{eq:oddity_formula} in the case of a discriminant form (cf.~e.g.~\cite[pp.\ 371, 383]{MR1662447})
is based on the orthogonality of the Jordan components and \emph{Milgram's formula}:
\begin{equation}
\frac{1}{\sqrt{\left|D\right|}}\sum_{\mu\in D}e\big(Q(\mu)\big)=e_{8}\big(\sign(\fqmQ)\big),\label{eq:milgrams_formula}.
\end{equation}
A proof of this formula is given by Milnor and Husemoller \cite[Appendix 4]{MR0506372}.
In our case, we instead use Lemma \ref{lem:gauss_sum_eq_gamma_p} together with the definition of signature to show that \eqref{eq:milgrams_formula} holds 
for any FQM.

The \emph{level} of a Jordan component $J$ is the smallest
positive integer $l$ such that the restriction of $Q$ to $J$ satisfies  $lQ\equiv0$ (mod $1$). % on this component. 
From Lemma \ref{lem:def-of-component} we see immediately that the level
of an odd $2$-adic component $\left[q_{t}^{\epsilon n}\right]$ is
$2q$ and that of any other component $\left[q^{\epsilon n}\right]$
is exactly $q$. The level of an FQM is the product of the
levels of all its Jordan components, and if the signature is odd then the level is divisible by $4$.
\begin{lemma}
\label{lem:x_c_def}
Let $(D,Q)$ be an FQM with level $l$, and let $c$ be an integer. Then there exists a unique element
$x_{c}\in D$ such that $D^{c*}=x_{c}+D^{c}$. Furthermore, this element
can be chosen as follows:
\begin{itemize}
\item If $2^{k}||c$ and $D$ has an odd $2$-adic Jordan
component $D\left[q_{t}^{\epsilon n}\right]$ with $q=2^{k}$ then
\[
x_{c}=2^{k-1}\sum\nolimits_{i}^{}\gamma_{i}, 
\]
where $\left\{ \gamma_{i}\right\} $
is an orthogonal basis of $D\left[q_{t}^{\epsilon n}\right]$. In this case $Q\left(x_{c}\right)\equiv\frac{ntq}{8}$ \textup{(mod $1$)}.
\item Otherwise we may take $x_{c}=0$.
\end{itemize}
In all cases: $2x_{c}=0$, $cx_{c}=0$ and if $\left(c,l\right)=\left(m,l\right)$
then $x_{c}=x_{m}$.
\end{lemma}
\begin{proof}
If $\alpha,\beta\in D^{c*}$ then $(\gamma,\alpha-\beta)\equiv cQ(\gamma)-cQ(\gamma)\equiv 0$ (mod $1$)
for any $\gamma\in D_{c}$. Hence $\alpha-\beta\in D_{c}^{\perp}=D^{c}$.
It follows that $D^{c*}/D^{c}$ has exactly one equivalence class,
say $\left[x_{c}\right]$ for some $x_{c}\in D^{c*}$. 
Let $k=\ord_2(c)$, that is, $2^{k}\parallel c$ and choose a fixed a Jordan decomposition of
$(D,Q)$ as above. Then $D\left[q^{\epsilon n}\right]\subseteq D_{q}\subseteq D_{c}$
for each $q|c$ and 
\[
D_{c}=\bigoplus_{2|q|c}D\left[q^{\epsilon n}\right]\bigoplus_{2<p|q|c}D\left[q^{\epsilon n}\right].
\]
Let $\left\{ \gamma_{i}\right\}$ be an orthonormal basis of the odd
component with $q=2^{k}$, where $\gamma_{i}$ is chosen as in Lemma
\ref{lem:def-of-component}. Note that $cQ\left(\gamma\right)\equiv0$ (mod $1$)
for $\gamma\in D\left[q^{\epsilon n}\right]$ with $q|c$ unless the
component is odd $2$-adic with $q=2^{k}$, in which case we instead
have $2cQ\left(\gamma\right)\equiv0$ (mod $1$). We now show that
$x_{c}=\sum_{i}a_{i}\gamma_{i}$ belongs to $D^{c*}$ if and only
if $cQ\left(\beta\right)+B\left(x_{c},\beta\right)=0$ for any $\beta\in D_{c}$.
It is enough to consider $\beta$ of the form $b\gamma_i$ and hence 
%in terms of the basis we see that this is equivalent to
\[
cb^{2}Q\left(\gamma_{i}\right)+2ba_{i}Q\left(\gamma_{i}\right)\equiv\frac{bt}{2q}\left(cb+2a_{i}\right)\equiv0 \text{ (mod $1$)}
\]
must hold for each $i$ and each integer $b$. It follows that we can choose $a_{i}=\frac{1}{2}q=2^{k-1}$
and then $x_{c}$ clearly has the required form. If $x_{c}\ne0$ it is 
easy to compute the norm and to verify that the remaining properties
are satisfied. 
\end{proof}
\begin{lemma}[{\cite[Prop.~2.2]{scheithauer:weil_rep}}]
\label{lem:Q_c_def}%If $\alpha\in D^{c*}$ then 
Let $(D,Q)$ be an FQM with level $l$, and let $c$ be an integer. Then the map $Q_c:D^{c*}\rightarrow \QQ/\ZZ$, which is defined by 
\[
Q_{c}(\alpha)  = cQ(y)+B(x_{c},y)\textup{ (mod $1$)},
\]
where $y\in D$ is given by $cy=\alpha-x_{c}$, is well-defined.
Furthermore, $cQ_{c}(\alpha)\equiv Q(\alpha)-Q(x_{c})$ \textup{(mod $1$)}.
 In particular, if $(c,l)=1$
then $Q_{c}(\alpha)\equiv \bar{c}Q(\alpha)$ \textup{(mod $1$)}
where $c\bar{c}\equiv1$ \textup{(mod $l$)} and if $l|c$ then $Q_{c}(\alpha)\equiv0$ \textup{(mod $1$)}.
\end{lemma}

\section{Evaluation of Gauss sums}
\label{sec:gaussums}
%An important part of our main result is 
An important step towards the explicit formulas for the Weil representation is the evaluation of Gauss sums of the type $\sum_{\mu\in D}e(Q(\mu))$. 
The different components in a Jordan decomposition are orthogonal with respect to the bilinear form. 
Hence the corresponding Gauss sum separates into local factors of the form 
\begin{equation}
g(J)=\frac{1}{\sqrt{|J|}}\sum_{\mu\in J}e\big(Q(\mu)\big)\label{eq:gfactor}. 
\end{equation}
Using a local version of Milgram's formula \eqref{eq:milgrams_formula} these sums can be expressed using the $\pexcess$ and oddity.
However, it is hard to find an easily accessible proof of this result (cf.~e.g.~\cite[Prop.\ 3.1]{scheithauer:weil_rep}) in the literature. 
To keep the exposition self-contained we spend the remainder of this subsection proving the following lemma.
\begin{lemma}
\label{lem:gauss_sum_eq_gamma_p}Let $\fqmQ$ be an FQM and
let $J$ be a Jordan component of $\fqmQ$ of order $q=p^{k}$ where 
$p$ is a prime and $k\ge0$. Then $g(J)=\gammafak_{p}(J)$ where $\gammafak_{p}(J)$ is defined by \eqref{eq:def:gamma_p}.
\end{lemma}
\begin{proof}
Write $J$ as a sum of indecomposables and use Lemmas
\ref{lem:g_q=00003Dgamma_q_pgt2}, \ref{lem:g_q=00003Dgamma_q_odd_2}
and \ref{lem:g_q=00003Dgamma_q_even2}.
\end{proof}
\begin{definition}
For $a,b\in\QQ$ and $c\in\ZZ$ we define $G\left(a,b,c\right)$ as %the following Gauss sum: 
\[
G\left(a,b,c\right)=\sum_{n \text{ (mod $c$)}}e_{c}\left(an^{2}+bn\right).
\]
In the notation of Berndt, Evans and Williams \cite{MR1625181} we have $G(a,b,c)=S(2a,2b,c)$.
Observe that all Gauss sums which we need can be expressed in terms of $G(a,b,c)$. 
\end{definition}
\begin{lemma}[{\cite[Thm.\ 1.2.2]{MR1625181}}]
\label{lem:Gauss_recip}Let $a,b,c\in\ZZ$ with $ac\ne0$.
Then \[
G\left(a,b,c\right)=\left|\frac{c}{2a}\right|^{\frac{1}{2}}e_{8}\left(\sign\left(2ac\right)-\frac{2b^{2}}{ac}\right)G\left(-\frac{c}{2},-b,2a\right).\]
\end{lemma}
\begin{corollary}
\label{cor:Gauss_even_odd}Let $b\in\ZZ$ and suppose that
$c>0$ is even. Then \[
G\left(1,b,c\right)=
\begin{cases}
0 & \text{if }\frac{c}{2}+b \text{ is odd},\\
\sqrt{2c}\, e_{8}\left(1-\frac{2b^{2}}{c}\right) & \text{if }\frac{c}{2}+b\,\,\text{ is even.}
\end{cases}\]
\end{corollary}
\begin{proof}
By Lemma \ref{lem:Gauss_recip} we know that $G(1,b,c) =\sqrt{\frac{c}{2}}e_{8}\big(1-\frac{2b^{2}}{c}\big)G(-\frac{c}{2},-b,2)$
and 
\[
G\left(-\frac{c}{2},-b,2\right)=\sum_{n=0}^{1}e\left(-\frac{1}{2}\left[\frac{c}{2}n^{2}+nb\right]\right)=1+\left(-1\right)^{\frac{c}{2}+b}.
\]
\end{proof}
\begin{lemma}
\label{lem:Gauss_a_0_c2}Let $a$ and $c$ be integers and suppose that $a$ is odd. Then 
\begin{align*}
G(a,0,2) & =1+\left(-1\right)^{a},\\
G(a,0,c) & =\sqrt{2c}\left(\frac{a}{2c}\right)e_8(a)\quad\text{if }\,c=2^{k}\text{ with }k>1,\\
G(a,0,c) & =\sqrt{c}\left(\frac{a}{c}\right)\varepsilon_{a}\quad\text{if $c>0$ is odd},
\end{align*}
where $\varepsilon_a=1$ if $a\equiv1$ \textup{(mod $4$)} and $\varepsilon_a=i$ otherwise.
\end{lemma}
\begin{proof}
The first equality is trivial, the second and third follows from \cite[Prop.\ 1.5.3 and Thm.~1.5.2]{MR1625181} using
 $1+i^{a}=\left(\frac{2}{a}\right)e_8(a)$.
\end{proof}
\begin{lemma}
\label{lem:g_q=00003Dgamma_q_pgt2}Let $(D,Q)$ be an FQM and $p$ an odd prime. If $q^\epsilon$ is an indecomposable, non-trivial $p$-adic Jordan component, 
which is generated by $\gamma$, then
\[
g\left(q^{\epsilon}\right)=\gammafak_{p}\left(q^{\epsilon}\right)=\left(\frac{2a}{q}\right)e_{8}\left(q-1\right),
\]
where $a$ is the unique integer \textup{(mod $p$)} that satisfies $(\frac{2a}{p})=\epsilon$
and $Q(\gamma)\equiv\frac{a}{q}$ \textup{(mod $1$)}. %for agenerator $\gamma$ of $D\left[q^{\epsilon}\right]$. 
\end{lemma}
\begin{proof}
From Lemma \ref{lem:def-of-component} we see that $q\gamma=0$ and $Q(\gamma)\equiv\frac{a}{q}$ \textup{(mod $1$)}, where $a$ is as described.
Since $e(Q(n\gamma))=e_q(an^{2})$ it follows that 
$\sqrt{q}\,g(q^{\epsilon})=G(a,0,q)$. 
By  Lemma \ref{lem:Gauss_a_0_c2}, using \eqref{eq:eps_d_formula} to rewrite $\varepsilon_a$, we see that $g(q^{\epsilon})=(\frac{2a}{q})e_{8}(q-1)$.
To see that this agrees with $\gammafak_p(q^{\epsilon})$ we observe that $e_8(4k)=(-1)^k$ and it is easy to verify that 
$(\frac{2a}{q})$ is equal to $-1$ precisely if $q$ is not a square and $\epsilon=-1$. 
\end{proof}
\begin{lemma}
\label{lem:g_q=00003Dgamma_q_odd_2}
Let $(D,Q)$ be an FQM. If $q_{t}^{\epsilon}$ is an indecomposable, nontrivial 
odd $2$-adic Jordan component then 
\[
g\left(q_{t}^{\epsilon}\right)=\gammafak_{2}\left(q_{t}^{\epsilon}\right)=\left(\frac{t}{q}\right)e_{8}\left(t\right).
\]
\end{lemma}
\begin{proof}
From Lemma \ref{lem:def-of-component} we know that $q_{t}^{\epsilon}$ is generated by an element $\gamma\in D$
satisfying $q\gamma=0$ and $Q(\gamma)=\frac{t}{2q}$ \textup{(mod $1$)}. Since $e(Q(n\gamma))=e_{2q}(n^{2}t)$
it follows that $\sqrt{q}g\left(q_{t}^{\epsilon}\right)=\sum_{n\text{ (mod $q$)}}e_{2q}(n^{2}t)$. This is an incomplete Gauss sum 
and we express the complete sum $G(t,0,2q)$ in two ways. By Lemma \ref{lem:Gauss_a_0_c2} we see that $G(t,0,2q)=2\sqrt{q}\left(\frac{t}{q}\right)e_{8}(t)$
and a direct computation shows that 
\[
G(t,0,2q)=\sum_{n\text{ (mod $q$)}}e_{2q}(n^{2}t)+\sum_{n\text{ (mod $q$)}}e_{2q}((n+q)^{2}t)=2\sum_{n\text{ (mod $q$)}}e_{2q}(n^{2}t).
\]
Hence $g(q_{t}^{\epsilon})=(\frac{t}{q})e_{8}(t)$, and we see that this agrees with $\gammafak_{2}(q_{t}^{\epsilon})$ in the same manner as in the proof of Lemma \ref{lem:g_q=00003Dgamma_q_pgt2}.
\end{proof}
\begin{lemma}
\label{lem:g_q=00003Dgamma_q_even2}Let $(D,Q)$ be an FQM. If $q^{2\epsilon}$ is an indecomposable, nontrivial 
even $2$-adic Jordan component then 
\[
g\left(q^{2\epsilon}\right)=\gammafak_{2}\left(q^{2\epsilon}\right)=\left(\frac{2-\epsilon}{q}\right).
\]
\end{lemma}
\begin{proof}
From Lemma \ref{lem:def-of-component} we know that $q^{2\epsilon}$ is generated by $\gamma$
and $\delta$ of order $q$, satisfying $B(\gamma,\delta)\equiv\frac{1}{q}$ (mod $1$)
and if $\epsilon=1$ then $Q(\gamma)=Q(\delta)=0$,
while if $\epsilon=-1$ then $Q(\gamma)=Q(\delta)\equiv\frac{1}{q}$ (mod $1$).

Consider first the case $\epsilon=1$. Then $\oddity(q^{2\epsilon})=0$
and $Q(a\gamma+b\delta)=a^2Q(\gamma)+b^2Q(\delta)+ab B(\gamma,\delta)\equiv \frac{ab}{q}$ (mod $1$). Hence
$qg\left(q^{2\epsilon}\right) =\sum_{a,b\text{ (mod $q$)}}e_{q}(ab)=q$.

In the case $\epsilon=-1$ we see that
$Q(a\gamma+b\delta)\equiv \frac{1}{q}(a^2+b^2+ab)$ (mod $1$) and
\[
  qg(q^{2\epsilon})=\sum_{a\text{ (mod $q$)}}e_{q}(a^{2})\sum_{b\text{ (mod $q$)}}e_{q}(b^{2}+ab) 
 =\sum_{a\text{ (mod $q$)}}e_{q}(a^{2})G(1,a,q).
\]
We need to consider the cases $q=2$ and $q\ne2$ separately. If $q=2$ then 
\[
2g(q^{2\epsilon})=G(1,0,2)+e_2(1)G(1,1,2)=\sum_{a\text{ (mod $1$)}}^{}e_{2}(a^{2})-\sum_{a \text{ (mod $1$)}}^{}e_{2}(a^{2}+a)=-2.
\]
If $q>2$ then $G(1,a,q)=0$ unless $a$ is even, and then $G(1,a,q)=\sqrt{2q}e_{8}(1-2a^{2}/q)$.
The sum above can therefore be restricted to the even elements and we get that
\[
qg\left(q^{2\epsilon}\right)=\sqrt{2q}e_{8}\left(1\right)\sum_{a \text{ (mod $q/2$)}}e_q(4b^{2}-b^{2})=\sqrt{2q}e_{8}(1)\sum_{a \text{ (mod $q/2$)}}e_q(3b^{2}).
\]
By Lemma \ref{lem:Gauss_a_0_c2} we  know that $G(3,0,q)=\sqrt{2q}(\frac{3}{q})e_{8}\left(-1\right)$,
and by using the fact that $3(a+q/2)^2\equiv3a$ (mod $q$) we see that $G(3,0,q)=2\sum_{a \text{ (mod $q/2$)}}e_q(3b^{2})$.
Hence $g(q^{2\epsilon})=(\frac{3}{q})$ and we conclude that
 $g(q^{2\epsilon})=(\frac{2-\epsilon}{q})=\gammafak_{2}(q^{2\epsilon})$ in all cases.
\end{proof}

The following lemma is the most important step in Scheithauer's \cite{scheithauer:weil_rep} 
proof of the explicit formula for the Weil representation in the case of even signature.

\begin{lemma}[{\cite[Thm.\ 3.9]{scheithauer:weil_rep}}]
\label{lem:Scheithauer3.9}Let $(D,Q)$ be
an FQM with associated bilinear form $B$. If $c$ is a nonzero integer and $\alpha\in D^{c*}$
then 
\[
\frac{1}{\sqrt{\left|D\right|}}\sum_{\mu\in D}e\big(cQ\left(\mu\right)+B\left(\alpha,\mu\right)\big)
=\sqrt{\left|D_{c}\right|}\,\Scheps{c}\,e\big(-Q_{c}\left(\alpha\right)\big),
\]
and otherwise the left hand side is equal to zero. Here 
\begin{align*}
\Scheps{c} =&  \prod_{2|q\nmid c}\gammafak_{2}\big(\left[q/q_{c}\right]_{*}^{\epsilon_{q}n_{q}}\big)
     e_{8}\big(\left(c/q_{c}-1\right)\oddity\left(\left[q/q_{c}\right]_{*}^{\epsilon_{q}n_{q}}\right)\big)\left(\frac{c/q_{c}}{\left(q/q_{c}\right)^{n_{q}}}\right) \times\\
   &\times\prod_{2<p|q\nmid c}\gammafak_{p}\big(\left[q/q_{c}\right]^{\epsilon_{q}n_{q}}\big)\left(\frac{c/q_{c}}{\left(q/q_{c}\right)^{n_{q}}}\right),
\end{align*}
where $q_{c}=\left(c,q\right)$ and where $*=t_{q}$ if the $2$-adic
component of order $q$ is odd, and $Q_c$ is given by Lemma \ref{lem:Q_c_def}. Note that the parameters $\epsilon_{q},n_{q}$
and $t$ are all the same as for $(D,Q)$. In particular,
if $\left(c,l\right)=1$ then 
\[
\Scheps{c}=e_{8}\big(\sign\left(\fqmQ\right)\big)\left(\frac{c}{\left|D\right|}\right)e_{8}\big(\left(c-1\right)\oddity\left(\fqmQ\right)\big).
\]
\end{lemma}
\begin{proof}
The proof is identical to that of Scheithauer \cite[Thm.\ 3.9]{scheithauer:weil_rep}
and we chose not to repeat the details here. The key point
is that $x_{c}$ is orthogonal to all components except
the odd $2$-adic component $D\left[q_{t}^{n\epsilon}\right]$ with
$q=2^{k}$ and $2^{k}||c$ (if it exists), where instead $cQ\left(\gamma\right)+B\left(x_{c},\gamma\right)=0$
for all $\gamma\in D\left[q_{t}^{n\epsilon}\right]$. The left hand
side therefore reduces to a product of local sums $\sum_{\mu\in J}e(cQ(\mu))$
over the other (i.e. $q\ne2^{k}$ or $q=2^{k}$ and of even type) Jordan
components, and these sums are then evaluated using the previous lemmas of this section.
\end{proof}

\section{The metaplectic group}
\label{sec:The-Metaplectic-Group}
Weil representations associated to finite quadratic
modules are often not representations of $\SLZ$, but of $\MPZ$, the metaplectic (two-fold) cover of $\SLZ$. 
We therefore repeat the most important facts about $\MPZ$. Additional details is given by e.g.~Gelbart \cite{MR0424695}. For $\mat{M}=\left(\begin{smallmatrix}a & b\\
c & d\end{smallmatrix}\right)\in\SLZ$ and $\tau\in\H$ define $j_{\mat{M}}\left(\tau\right)=c\tau+d$, 
\[
\sigma_{\mat{M}}=
\begin{cases} 
 c & \text{if }\, c\ne 0,\\
 d & \text{otherwise},
\end{cases}
 \quad
c_{\mat{M}}=c,\text{ and } 
d_{\mat{M}}=d.
\]
It is well-known that $\MPZ$ can be realized as the group of pairs 
$
\big(\mat{M},\varphi_{\mat{M}}(\tau)\big),
$
where $\varphi_{\mat{M}}(\tau)$ is a solution to $\varphi_{\mat{M}}(\tau)^{2}=j_{\mat{M}}(\tau)$,
with the following group multiplication:
\begin{equation}
\left(\mat{A},\varphi_{\mat{A}}(\tau)\right)\left(\mat{B},\varphi_{\mat{B}}\left(\tau\right)\right)=\left(\mat{AB},\varphi_{\mat{A}}\left(\mat{B}\tau\right)\varphi_{\mat{B}}
(\tau)\right).\label{eq:Mp2Zmult-1}
\end{equation}
For $\mat{M}\in\SLZ$ we define $\widetilde{\mat{M}}=(\mat{M},\sqrt{j_{\mat{M}}(\tau)})$
to be the canonical choice of representative in the inverse image
of $\mat{M}$ under the covering map. Using the generators
$\mat{S}=\left(\begin{smallmatrix}0 & -1\\ 1 & 0\end{smallmatrix}\right)$ and 
$\mat{T}=\left(\begin{smallmatrix}1 & 1\\ 0 & 1\end{smallmatrix}\right)$
of $\SLZ$ it
is easy to see that $\MPZ$ is generated by 
\[
\widetilde{\mat{T}}=\left(\left(\begin{smallmatrix}1 & 1\\
0 & 1\end{smallmatrix}\right),1\right)\quad\text{and}
\quad\widetilde{\mat{S}}=\left(\left(\begin{smallmatrix}0 & -1\\
1 & 0\end{smallmatrix}\right),\sqrt{\tau}\right),
\]
where the only relations are 
$\widetilde{\mat{S}}^{2}=(\widetilde{\mat{ST}})^{3}
=\widetilde{\mat{Z}}$. Here $\widetilde{\mat{Z}}$ is the generator of the center of $\MPZ$. Explicitly, this element is given by
\[
\widetilde{\mat{Z}}=\left(\left(\begin{smallmatrix}-1 & 0\\
0 & -1\end{smallmatrix}\right),i\right),
\quad\widetilde{\mat{Z}}^{2}=\left(\left(\begin{smallmatrix}1 & 0\\
0 & 1\end{smallmatrix}\right),-1\right)\quad\text{and}\quad\widetilde{\mat{Z}}^{4}=\left(\left(\begin{smallmatrix}1 & 0\\
0 & 1\end{smallmatrix}\right),1\right)=\textup{id}_{\MPZ}.
\]
By Kubota \cite{MR0204422} it is known
that the metaplectic cover of $\SLZ$ is determined by the $2$-cocycle
$\mu:\SLZ^{2}\rightarrow\left\{ \pm1\right\} $ defined by 
\[
\mu\left(\mat{A},\mat{B}\right)=\left(\sigma_{\mat{A}}\sigma_{\mat{AB}},\sigma_{\mat{B}}\sigma_{\mat{AB}}\right)_{\infty}.
\]
However, if we study the metaplectic group from the point of
view of half-integral weight modular forms (cf.~e.g. Shimura \cite{shimura:73:half_integral}),
we are instead naturally led to the $2$-cocycle $\sigma:\SLZ^{2}\rightarrow\left\{ \pm1\right\}$,
defined for $\mat{A},\mat{B}\in\SLZ$ by 
\[
\widetilde{\mat{A}}\widetilde{\mat{B}}=(\mat{AB},\sigma\left(\mat{A},\mat{B}\right)\sqrt{j_{\mat{AB}}\left(\tau\right)}).
\]
More precisely, for any choice of $\tau\in\H$ we have 
\begin{align}
\sigma\left(\mat{A},\mat{B}\right) 
= & j_{\mat{A}}\left(\mat{B}\tau\right)j_{\mat{B}}\left(\tau\right)j_{\mat{AB}}(\tau)^{-1}=e^{\pi iw(\mat{A},\mat{B})},\,\text{where}\label{eq:sigma-def}\\ 
  w(\mat{A},\mat{B})  
=&  \frac{1}{2\pi}\big(\Arg j_{\mat{A}}(\mat{B}\tau)+\Arg j_{\mat{B}}(\tau)-\Arg j_{\mat{AB}}(\tau)\big)\in\left\{ 0,\pm1\right\} .\nonumber 
\end{align}
Let $\Gamma$ be a subgroup of $\SLZ$ and let $\widetilde{\Gamma}\subseteq\MPZ$
denote the inverse image of $\Gamma$ under the covering map. If $\widetilde{\chi}:\widetilde{\Gamma}\rightarrow S^{1}$
is a character of $\widetilde{\Gamma}$ then 
$\widetilde{\chi}\big((\mat{A},-\sqrt{j_{\mat{A}}(\tau)})\big)=\epsilon\widetilde{\chi}(\widetilde{\mat{A}})$
with $\epsilon=\widetilde{\chi}(\widetilde{\mat{Z}}^{2})$, and if we define
$\chi:\Gamma\rightarrow S^{1}$ by $\chi\left(\mat{A}\right)=\widetilde{\chi}(\widetilde{\mat{A}})$
 we say that $\chi$ is \emph{induced} by $\widetilde{\chi}$ and 
\[
\chi\left(\mat{A}\right)\chi\left(\mat{B}\right)=\alpha\left(\mat{A},\mat{B}\right)\chi\left(\mat{A}\mat{B}\right),
\]
where $\alpha\left(\mat{A},\mat{B}\right)=\epsilon$ if $\sigma\left(\mat{A},\mat{B}\right)=-1$
and $1$ otherwise. Since $\widetilde{\mat{Z}}^{4}$ is the identity we
see that $\epsilon=\pm1$, and if $\epsilon=1$ then $\chi$ is a character,
meaning that $\widetilde{\chi}$ factors through $\Gamma$. If $\epsilon=-1$
then $\alpha\left(\mat{A},\mat{B}\right)=\sigma\left(\mat{A},\mat{B}\right)$ and $\chi$
is a projective character, which we can view as a half-integral weight
multiplier system on $\Gamma$. Cf.~e.g.~\cite[Part I]{MR1513171},
\cite[Ch.~3]{maass:modular_functions}, \cite[Ch.\ 3]{rankin:mod}
or \cite[Ch.\ VI, \S 5]{MR2513384}. 

The construction of induced characters outlined here works analogously for representations 
on higher dimensional vector spaces. 
More information on projective
representations in general is given by Mackey \cite{MR0098328}. In
particular, we say that a representation $\widetilde{\rho}:\MPZ \rightarrow\mathbb{C}^{n}$
induces a vector-valued multiplier system of half-integral weight
if there exists a character $\widetilde{\chi}$ of $\MPZ$ such that
$\widetilde{\chi}^{-1}\widetilde{\rho}$ induces an $n$-dimensional
unitary representation of $\SLZ$. Cf.~\cite[Ch.~9]{hejhal:lnm1001}.

Although we know by the theorem of Kubota \cite{MR0204422} that $\sigma$ and $\mu$ must be related
through a trivial cocycle we could not find an explicit formula for it in the literature. 
To clarify this relationship we provide Theorem \ref{thm:sigmaAB} below.
For explicit computations the following expression for $\sigma$ in terms
of Hilbert symbols is useful. 
\begin{lemma}
\label{lem:sigma(A,B)}If $A,\mat{B}\in\SLZ$ and $\mat{C}=\mat{AB}$ then 
\[
\sigma\left(\mat{A},\mat{B}\right)=
\begin{cases}
\left(\sigma_{\mat{C}}\sigma_{\mat{A}},\sigma_{\mat{C}}\sigma_{\mat{B}}\right)_{\infty} & \text{if } c_{\mat{A}}c_{\mat{B}}c_{\mat{C}}\ne0,\\
\left(\sigma_{\mat{A}},\sigma_{\mat{B}}\right)_{\infty} & \text{if }c_{\mat{A}}c_{\mat{B}}\ne0\text{ and } c_{\mat{C}}=0,\,\text{or }c_{\mat{A}}=c_{\mat{B}}=0,\\
\left(-\sigma_{\mat{A}},\sigma_{\mat{B}}\right)_{\infty} & \text{if }c_{\mat{A}}\ne0,\, c_{\mat{B}}=0,\\
\left(-\sigma_{\mat{B}},\sigma_{\mat{A}}\right)_{\infty} & \text{if }c_{\mat{A}}=0,\, c_{\mat{B}}\ne0.
\end{cases}
\]
\end{lemma}
\begin{proof}
The key ingredient in the proof is to compute the various arguments
by setting $z=iy$ and letting $y\rightarrow\infty$. 
An analogous statement for the function $w(\mat{A},\mat{B})$  was proven by Maa{\ss}~
\cite[p.\ 115]{maass:modular_functions}.
\end{proof}
We can now easily prove the following
theorem relating $\sigma$ and $\mu$. 
\begin{theorem}
\label{thm:sigmaAB}For any $\mat{A},\mat{B}\in\SLZ$ we can express  $\sigma(\mat{A},\mat{B})$ as
\[
\sigma(\mat{A},\mat{B})=\mu(\mat{A},\mat{B})\frac{s(\mat{A})s(\mat{B})}{s(\mat{AB})},
\]
%where $s:\SLZ\rightarrow\left\{ \pm1\right\} $ is given by
where $s(\mat{A})=1$
if $c_{\mat{A}}\ne0$ and $s(\mat{A})=\sign(d_{A})$
if $c_{\mat{A}}=0$.
\end{theorem}
Note that the expression $s(\mat{A})s(\mat{B})s(\mat{AB})^{-1}:\SLZ\rightarrow \{\pm1\}$
is a trivial 2-cocycle.  The above theorem therefore gives a
constructive proof of the fact that the two 2-cocycles $\sigma$ and
$\mu$ are equivalent. 

To prove the explicit formula in our main theorem we only need to
evaluate $\sigma(\mat{A},\mat{B})$ for certain matrices $\mat{A}$ and $\mat{B}$.
A short calculation using Lemma \ref{lem:sigma(A,B)} proves the following two lemmas.
\begin{lemma}
\label{lem:sigma(A,BTm)}Let $\mat{A},\mat{B}\in\SLZ$ and $m\in\ZZ.$
Then we have $\sigma\left(\mat{A},\mat{T}^{m}\right)=\sigma\left(\mat{T}^{m},\mat{A}\right)=1$ 
and
$\sigma\left(\mat{A},\mat{B}\mat{T}^{m}\right)=\sigma\left(\mat{A},\mat{B}\right)=\sigma\left(\mat{T}^{m}\mat{A},\mat{B}\right)$. 
\end{lemma}
\begin{lemma}
\label{lem:sigma(A,S)}Let $\mat{A}=\left(\begin{smallmatrix}a & b\\
c & d\end{smallmatrix}\right)\in\SLZ$. Then
% and $\mat{S}=\left(\begin{smallmatrix}0 & -1\\
%1 & 0\end{smallmatrix}\right)$. Then
\[
\sigma\left(\mat{A},\mat{S}\right)=\epsilon_{c,d}:=
\begin{cases}
\left(-c,d\right)_{\infty} & \text{if } c\ne0,\\
\left(-1,d\right)_{\infty} & \text{if } c=0,

\end{cases}=
\begin{cases}
-1 & \text{if } c\ge0\,\text{ and } d<0,\\
1 & \text{otherwise}.
\end{cases}
\]
\end{lemma}
\begin{lemma}
\label{lem:sigma(STmSTn)}Let $m$ and $n$ be integers. Then
\[
\sigma\left(\mat{ST}^{m},\mat{ST}^{n}\right)=
\begin{cases}
-1 & \text{if } m<0,\\
1 & \text{otherwise}.
\end{cases}
\]
\end{lemma}
\begin{proof}
Let $\mat{A}=\mat{ST}^{m}$, $\mat{B}=\mat{ST}^{n}$ and $\mat{C}=\mat{AB}$. Observe that $\sigma_{\mat{A}}=\sigma_{\mat{B}}=1$.
If $m\ne0$ then $\sigma_{\mat{C}}=m$ and $\sigma(\mat{A},\mat{B})=(\sigma_{\mat{C}},\sigma_{\mat{A}}\sigma_{\mat{B}}\sigma_{\mat{C}})_{\infty}=
(m,m)_{\infty}=-1$
if $m<0$ and $=1$ if $m>0$. If $m=0$ then $c_{\mat{C}}=0$, $\sigma_{\mat{C}}=-1$
and $\sigma(\mat{A},\mat{B})=(\sigma_{\mat{A}},\sigma_{\mat{B}})_{\infty}=(1,1)_{\infty}=1$. 
\end{proof}
\begin{lemma}
\label{lem:sigma(xStmSTn)=00003D1}Let $\mat{M}=\left(\begin{smallmatrix}a & b\\
c & d\end{smallmatrix}\right)\in\SLZ$ with $c>0$. If $m$ and $n$ are integers satisfying $cn-d>0$ and $m>0$ then 
\[
\sigma\left(\mat{M}\mat{T}^{-n}\mat{S}\mat{T}^{-m}\mat{S},\mat{S}\mat{T}^{m}\mat{ST}^{n}\right)=1.
\]
\end{lemma}
\begin{proof}
Let $\mat{A}=\mat{MT}^{-n}\mat{ST}^{-m}\mat{S}$, $\mat{B}=\mat{ST}^{m}\mat{ST}^{n}$ and $\mat{C}=\mat{AB}=\mat{M}$. 
Then $\sigma_{\mat{B}}=c_{\mat{B}}=m>0,$
$\sigma_{\mat{A}}=c_{\mat{A}}=c>0$ and by Lemma \ref{lem:sigma(A,B)} we see
that 
\[
\sigma\left(\mat{X},\mat{ST}^{m}\mat{ST}^{n}\right)=
\begin{cases}
\left(\sigma_{\mat{C}}\sigma_{\mat{A}},\sigma_{\mat{C}}\sigma_{\mat{A}}\right)_{\infty} & \text{if } c_{\mat{A}}\ne0,\\
\left(-\sigma_{\mat{B}},\sigma_{\mat{A}}\right)_{\infty} & \text{ if } c_{\mat{A}}=0.
\end{cases}
\]
Then $(\sigma_{\mat{C}}\sigma_{\mat{A}},\sigma_{\mat{C}}\sigma_{\mat{B}})_{\infty}=(c\sigma_{\mat{A}},cm)_{\infty}=
\left(\sigma_{\mat{A}},m\right)_{\infty}=1$
since $m>0$, and if $c_{\mat{A}}=0$ then $\left(-\sigma_{\mat{B}},\sigma_{\mat{A}}\right)$
$=\left(-m,cn-d\right)_{\infty}=1$ since $cn-d>0$. 
\end{proof}

\section{The Weil representation}
\label{sec:The-Weil-Representation}
If $\fqmQ$ is an FQM and $\fqmQ=(D,Q)$ then the Weil representation associated to $\fqmQ$
is a unitary finite-dimensional representation of $\MPZ$ on the group
algebra of $D$, $\mathbb{C}\left[D\right]\simeq\mathbb{C}^{\left|D\right|}.$
We use $\vec{e}_{\gamma},$ $\gamma\in D$ to denote the basis
vectors, and $\textup{id}_{\mathbb{C}\left[D\right]}$ the identity
element, of $\mathbb{C}\left[D\right]$. 
For the Weil representations of the type we consider here we provide the most important results, with an emphasis on explicit formulas.   
A more through theoretical background is given by e.g~Skoruppa \cite{skoruppa-weilrep} or Gelbart \cite{MR0424695}. 
\subsection{Definition and properties of the Weil representation}
In the remainder of this section let $\fqmQ$ be a given FQM with $\fqmQ=(D,Q)$ and associated bilinear form $B$. 
Assume in addition, for simplicity, that we have chosen a fixed Jordan decomposition of $\fqmQ$. 
\begin{definition}
\label{def:weil-rep}The Weil representation $\rt:\MPZ\rightarrow\mathbb{C}\left[D\right]$
associated to $\fqmQ$ is defined by the following action of the generators. If $\gamma\in D$ then 
\begin{align*}
\rt\big(\widetilde{\mat{T}}\big)\vec{e}_{\gamma} & =e\big(Q(\gamma)\big)\vec{e}_{\gamma},\\
\rt\big(\widetilde{\mat{S}}\big)\vec{e}_{\gamma} & =\si\frac{1}{\sqrt{\left|D\right|}}\sum_{\delta\in D}e\big(-B(\gamma,\delta)\big)\vec{e}_{\delta}.
\end{align*}
Here $\si$ is an eighth-root of unity, called the \emph{Witt-invariant} of $\fqmQ$, defined by   
\[
\si=\frac{1}{\sqrt{\left|D\right|}}\sum_{\gamma\in D}e\big(-Q(\gamma)\big).
\]
\end{definition}
Note that the Weil representation, as defined by Scheithauer \cite{scheithauer:weil_rep},
is in fact the dual of the one defined above. 
The action of the element 
 $\widetilde{\mat{Z}}=\widetilde{\mat{S}}^{2}$  is 
readily determined using the fact that $\sum_{\delta\in D}e\left(-B\left(\delta,\alpha\right)\right)=0$ unless 
$\alpha=0$, in which case it is equal to $|D|$. 
\begin{lemma}
\label{lem:action_of_Z}If $\gamma\in D$ then the element $\widetilde{\mat{Z}}$ acts as
\[
\rt(\widetilde{\mat{Z}})\vec{e}_{\gamma}=\si^{2}\vec{e}_{-\gamma}.
\]
\end{lemma}
Since $\widetilde{\mat{Z}}^{4}=1$ the lemma implies that $\si^{8}=1$,
and in particular $\si^{4}=\pm1$. 
\begin{lemma}
\label{lem:The-Witt-invariant}The Witt invariant $\si$
can be expressed by 
\[
\si=e_{8}\big(-\sign(\fqmQ)\big).
\]
\end{lemma}
\begin{proof}
This follows directly from Milgram's formula \eqref{eq:milgrams_formula}. 
\end{proof}
The action of the central element $\widetilde{\mat{Z}}$ is intimately connected to the signature. 
%If the signature of $\fqmQ$ is even then $\si^{4}=1$ and $\rt$ induces a unitary representation of $\SLZ$. 
%If the  
In fact, depending on whether the signature of $\fqmQ$ is even or odd we see that $\si^{4}=1$ or $-1$   
and that $\rt$ either induces a representation or a projective representation of $\SLZ$, respectively.
To treat both cases simultaneously we write the multiplicative relation for the induced (projective) representation $\r$ as
%Therefore the multiplicative relation for $\r$ is given by  
\begin{equation}
\r(\mat{A})\r(\mat{B})=\sigma_{\fqmQ}(\mat{A},\mat{B})\r(\mat{AB}),\label{eq:co-cycle-for-rho}
\end{equation}
where $\sigma_{\fqmQ}\left(\mat{A},\mat{B}\right)$ is defined as 
\[
\sigma_{\fqmQ}(\mat{A},\mat{B})=
\begin{cases}
1 & \text{if } \sign(\fqmQ)\equiv0\text{ (mod $2$)},\\
\sigma(\mat{A},B) & \text{if } \sign(\fqmQ)\equiv1\text{ (mod $2$)},
\end{cases}
\]
with $\sigma(\mat{A},\mat{B})$ defined in \eqref{eq:sigma-def}.
If the signature is odd then we view $\r$ as a vector-valued multiplier
system of half-integral weight on $\SLZ$. In this case, using Lemmas
\ref{lem:sigma(A,B)} and \ref{lem:action_of_Z}, it is an easy exercise to verify
the following lemma.
\begin{lemma}
\label{lem:rho(minusA)}
If $\mat{A}=\left(\begin{smallmatrix}a & b\\ c & d\end{smallmatrix}\right)\in\SLZ$
and the signature of $\fqmQ$ is odd then 
\[
\r\left(-\mat{A}\right)\vec{e}_{\gamma}=\mu\,\r\left(\mat{A}\right)\vec{e}_{-\gamma}\quad \text{for all }\,\gamma \in D,
\]
where
 \[
\mu=\si^{2}\cdot
\begin{cases}
 -\sign(c) & \text{if } c\ne0,\\
  \sign(d) & \text{if } c=0.
\end{cases}
\]
\end{lemma}
\begin{definition}
If $d$ is an integer we let $\fqmQ^{d}$ denote the FQM  with the same abelian group  as $\fqmQ$ but with quadratic form scaled by $d$, that is $\fqmQ^{d}=(D,dQ)$. 
\end{definition}
When we restrict the Weil representation $\r$ to certain congruence subgroups 
then the Witt invariant of $\fqmQ^{d}$ show up in the resulting formulas (cf.~Lemma \ref{lem:action_on_gamma_0^0}). We therefore need the following lemma.
\begin{lemma}
\label{sigma_Qd_eps_Qd}If the level of $\fqmQ$ is $l$ and $d$ is an integer relatively prime to $l$ then the Witt invariant
of $\fqmQ^{d}$ is given by 
\[
\sigma_{w}(\fqmQ^{d})=\left(\frac{d}{\left|D\right|}\right)e_{8}\big(-\sign(\fqmQ)+(1-d)\oddity(\fqmQ)\big),
\]
and thus $\eqd{d}$, defined as $\eqd{d}=\sigma_w(\fqmQ^{d}) \si^{-1}$, is given by 
\[
\eqd{d}=\left(\frac{d}{\left|D\right|}\right)e_{8}\big((1-d)\oddity(\fqmQ)\big).
\]
Furthermore, the symbol $\eqd{d}$ satisfies $\eqd{d'}=\eqd{d}$ for
any $d'\equiv d$ \textup{(mod $l$)}. 
\end{lemma}
\begin{proof}
Let $D_{d}$, $D^{d}$, $x_{d}$, $D^{d*}$ be as in Section \ref{sec:On-Jordan-Decompositions}.
If $\left(d,l\right)=1$ then $D_{d}=\left\{ 0\right\}$, $D^{d}=D$
and $x_{d}=0$ so that $D^{d*}=D$, and in particular $0\in D^{d*}$.
By Lemma \ref{lem:Scheithauer3.9} we see that
 \[
\frac{1}{\sqrt{\left|D\right|}}\sum_{\alpha\in D}e\big(dQ(\alpha)\big)=
e_{8}\big(\sign\left(\fqmQ\right)\big)\left(\frac{d}{\left|D\right|}\right)e_{8}\big(\left(d-1\right)\oddity\left(\fqmQ\right)\big).
\]
That $\eqd{d}$ only depends on $d$ modulo the level is clear from the definition.
\end{proof}
Using Lemma \ref{sigma_Qd_eps_Qd} and the observation that $\sigma_w(\fqmQ^{-1})=\overline{\sigma_w(\fqmQ^{1})}$
it is easy to show that the oddity and the signature are related through
\begin{align}
e_{4}\big(\oddity(\fqmQ)-\sign(\fqmQ)\big) &= \left(\frac{-1}{\left|D\right|}\right),\quad\text{or equivalently}\nonumber \\
\oddity(\fqmQ) - \sign(\fqmQ) &\equiv \left(\frac{-1}{\left|D\right|}\right)-1\textup{ (mod $4$)}\label{eq:formula_for_-1_over_D}.
\end{align}
%\begin{definition}
If $d$ is an odd integer we define $\varepsilon_{d}\in\left\{ 1,i\right\}$ by
\begin{equation}
\varepsilon_{d}=
\begin{cases}
1 & \text{ if } d\equiv1 \text{ (mod $4$)},\\
i & \text{ if } d\equiv3 \text{ (mod $4$)}.
\end{cases}\label{eq:epsilon_d}
\end{equation}
%\end{definition}
The symbol $\varepsilon_{d}$ is an essential part of the multiplier
system for the Jacobi theta function, and it is possible to express
it in (at least) two different ways in terms of Kronecker symbols
and exponentials. By verifying this identity for all odd $d$ (mod $8$) we see
 that if $d$ is an odd integer then 
\begin{equation}
\varepsilon_{d}=\left(\frac{2}{d}\right)e_{8}\left(1-d\right)=e_{8}\left(1-\left(\frac{-1}{d}\right)\right).\label{eq:eps_d_formula}
\end{equation}

\begin{lemma}
\label{lem:eps_Q_d_formula}If $\fqmQ$ has level $l$ and $d$ is an integer relatively prime to $l$ then
\[
\eqd{d}=\begin{cases}
\left(\frac{d}{\left|D\right|2^{\sign(\fqmQ)}}\right)\varepsilon_{d}^{\sign(\fqmQ)+\left(\frac{-1}{\left|D\right|}\right)-1} & \text{if } 4|l,\\
\left(\frac{d}{\left|D\right|}\right) & \text{if }4\nmid l.
\end{cases}
\]
Furthermore, if $4|l$ then 
\[
\eqd{-d}=\eqd{d}\,e_{4}\big(\left(d-1\right)\oddity\left(\fqmQ\right)+\sign\left(\fqmQ\right)\big).
\]
\end{lemma}
\begin{proof}
By Lemma \ref{sigma_Qd_eps_Qd} and \eqref{eq:eps_d_formula} we see
that if $d$ is an odd integer then 
\[
\eqd{d}=\left(\frac{d}{\left|D\right|}\right)\left[\left(\frac{2}{d}\right)\varepsilon_{d}\right]^{\oddity\left(\fqmQ\right)}
=\left(\frac{d}{\left|D\right|2^{\sign\left(\fqmQ\right)}}\right)\varepsilon_{d}^{\sign\left(\fqmQ\right)+\left(\frac{-1}{\left|D\right|}\right)-1}.
\]
If $4$ does not divide $l$ we know that the only possible $2$-adic
Jordan components are of even type. Therefore $\oddity\left(\fqmQ\right)\equiv0$ (mod $4$)
and $\eqd{d}=(\frac{d}{\left|D\right|})$. The last expression
follows from a simple computation using \eqref{eq:formula_for_-1_over_D}.
\end{proof}
The formula for $\eqd{d}$ in Lemma \ref{lem:eps_Q_d_formula} should be compared
with the corresponding formula of Borcherds \cite[Thm.\ 5.4]{MR1773561} (note that in his terminology
$\chi_{\theta}=\varepsilon_{d}^{-1}$).  We now use $\eqd{d}$
to prove the following properties of the quadratic residue symbol. 
\begin{corollary}
\label{cor:prop.of.kron}If $\fqmQ$ has level $l$ and $d$ and $k$ are integers
with $d$ relatively prime to $l$ then
\begin{align*}
\left(\frac{d+kl}{\left|D\right|}\right) & =\left(\frac{d}{\left|D\right|}\right)e_{8}\big(kl\,\oddity(\fqmQ)\big),\quad\text{and in particular}\\
\left(\frac{d+4Nk}{2N}\right) & =\left(\frac{d}{2N}\right)\left(-1\right)^{Nk}\quad\text{for all odd integers } N. 
\end{align*}
\end{corollary}
\begin{proof}
By Lemma \ref{sigma_Qd_eps_Qd} we see that if $\left(d,l\right)=1$
and $k\in\ZZ$ then $\eqd{d+kl}=\eqd{d}$ and
\[
e_{8}\big((1-(d+kl))\oddity(\fqmQ)\big)\left(\frac{d+kl}{\left|D\right|}\right)
=\left(\frac{d}{\left|D\right|}\right)e_{8}\big(\left(1-d\right)\oddity\left(\fqmQ\right)\big).
\]
The second statement, which is a special case of the first, corresponds
to the discriminant form $\fqmQ_N$ from Example \ref{exa:Nx^2}. It has level $l=4N$, 
 $\left|D\right|=2N$ and $\oddity(\fqmQ_N)\equiv N$ (mod $4$).
Note that this statement can also be verified directly using the fact
that $\left(\frac{d+4}{2}\right)=-\left(\frac{d}{2}\right)$ for any
integer $d$. 
\end{proof}
To reach the main goal of this paper, that is, to obtain an explicit formula
for the action of the full modular group on the Weil representation,
we extend known formulas for the action of $\Gamma_{0}^{0}\left(l\right)$,
in a first step to $\Gamma_{0}\left(l\right)$, and then in a second
step to the full modular group. This is the same approach as that taken by Scheithauer \cite{scheithauer:weil_rep}.
\begin{definition}
We need three congruence subgroups of level $l$, namely
% $\Gamma(l)\subset\Gamma_{0}^{0}(l)\subset\Gamma_{0}(l)$, 
the \emph{principal} congruence subgroup, $\Gamma(l)$, and the two groups $\Gamma_0^0(l)$ and $\Gamma_0(l)$. 
Here 
\[
\Gamma_{0}\left(l\right)=\left\{ \left(\begin{array}{cc}
a & b\\
c & d\end{array}\right)\in\SLZ\,\Big{|}\, c\equiv0\text{ (mod $l$)}\right\},  
\]
$\Gamma_{0}^{0}(l)\subset \Gamma_0(l)$ is the subgroup consisting of all matrices 
with $b\equiv 0$ (mod $l$) and $\Gamma(l)$ is the group consisting of all matrices in $\SLZ$ which are congruent to the $2\times2$ 
identity matrix mod $l$. 
\end{definition}
\begin{definition}
\label{def:eq}If $\mat{A}=\left(\begin{smallmatrix}a & b\\
c & d\end{smallmatrix}\right)\in\SLZ$ we define $\eq:\SLZ\rightarrow\left\{ \pm1\right\}$ by 
\[
\eq\left(\mat{A}\right)=
\begin{cases}
\left(\frac{c}{d}\right) & \text{if }\sign(\fqmQ)\,\text{ is odd},\\
1 & \text{otherwise,}
\end{cases}
\]
and then define
$\ch:\SLZ\rightarrow\left\{ e_{8}\left(k\right)\,|\, k\in\ZZ/8\ZZ\right\}$ by
\[
\ch\left(\mat{A}\right)=\eq\left(\mat{A}\right)\eqd{d}^{-1}.\]
\end{definition}
\begin{remark}
\label{rem:chi_as_v_theta}Recall that the Jacobi theta function
$\theta(\tau)=\sum_{n\in \ZZ} e(\tau n^2)$ is a weight $\frac{1}{2}$ modular form on $\Gamma_{0}\left(4\right)$ with the theta
multiplier system, given by (cf.~e.g.~\cite[\S 2]{shimura:73:half_integral})
\[
v_{\theta}\left(\mat{A}\right)=\left(\frac{c}{d}\right)\varepsilon_{d}^{-1},\quad \text{ for all  }\mat{A}=\left(\begin{smallmatrix}a & b\\
c & d\end{smallmatrix}\right)\in \Gamma_0(4).
\]
\end{remark}
By Corollary \ref{lem:eps_Q_d_formula} it follows that if $4|l$ and $\mat{A}=\left(\begin{smallmatrix}a & b\\
c & d\end{smallmatrix}\right)\in \Gamma_0(l)$ then 
\begin{align*}
\ch(\mat{A}) & =\left(\frac{d}{\left|D\right|2^{\sign(\fqmQ)}}\right)\left(\frac{c}{d}\right)^{\sign\left(\fqmQ\right)}\varepsilon_{d}^{-\sign\left(\fqmQ\right)+1-\left(\frac{-1}{\left|D\right|}\right)}\\
             & =\left(\frac{d}{\left|D\right|2^{\sign(\fqmQ)}}\right)\,\times
\begin{cases}
v_{\theta}(\mat{A})             & \text{if } \sign(\fqmQ)+(\frac{-1}{\left|D\right|})\equiv2\text{ (mod $4$)},\\
\overline{v_{\theta}}(\mat{A})  & \text{if } \sign(\fqmQ)+(\frac{-1}{\left|D\right|})\equiv0\text{ (mod $4$)},\\
\left(\frac{-1}{d}\right)       & \text{if } \sign(\fqmQ)+(\frac{-1}{\left|D\right|})\equiv3\text{ (mod $4$)},\\
1                               & \text{if } \sign(\fqmQ)+(\frac{-1}{\left|D\right|})\equiv1\text{ (mod $4$)}.
\end{cases}
\end{align*}
If $4\nmid l$ then the signature of $\fqmQ$ is even and it
is clear that $\ch(\mat{A})=(\frac{d}{\left|D\right|})$.
We conclude that if the signature of $\fqmQ$ is even then $\ch$
restricted to $\Gamma_{0}\left(l\right)$ equals a Dirichlet character,
while if the signature is odd then $\ch$ is identical to the half-integral
weight multiplier system given by the theta multiplier, $v_{\theta},$
or its conjugate, $\overline{v}_{\theta}$, times a Dirichlet character.
Observe that $\overline{v}_{\theta}=(\frac{-1}{d})v_\theta$.

The action of the subgroups $\Gamma\left(l\right),$ $\Gamma_{0}^{0}\left(l\right)$
and $\Gamma_{0}\left(l\right)$ on the Weil representation, as given
in Lemma \ref{lem:action_on_gamma_0^0} and \ref{lem:action_on_gamma_0} below, 
was obtained for certain discriminant forms already by Schoeneberg \cite{MR1513241}
(for even signature) and Pfetzer \cite{MR0059945} (for odd signature)
in terms of transformation formulas for theta functions corresponding
to integral lattices. In a more modern language, these results are also given by 
 e.g.~Ebeling \cite[Ch.\ 3.1]{MR1280458} (based
on lectures and notes of Hirzebruch and Skoruppa). Compare also with results by Borcherds \cite[Thm.\ 5.4]{MR1773561}
and Eichler \cite[p.~49]{MR0209258}. For an arbitrary FQM the
following lemma follows from results of Skoruppa \cite{skoruppa-weilrep}. 

\begin{lemma}
\label{lem:action_on_gamma_0^0}If $\fqmQ$ has level $l$ and  
 $\mat{A}=\left(\begin{smallmatrix}a & b\\ c & d\end{smallmatrix}\right)\in \SLZ$ then 
\[
\r\left(\mat{A}\right)\vec{e}_{\gamma}=
\begin{cases}
\eq\left(\mat{A}\right)\vec{e}_{\gamma} & \text{ if } \mat{A}\in\Gamma(l),\\
\ch\left(\mat{A}\right)\vec{e}_{d\gamma} & \text{ if } \mat{A}\in\Gamma_{0}^{0}(l).
\end{cases}
\]
\end{lemma}
To illustrate the technique which we use to extend the
formulas for the action of $\Gamma_{0}\left(l\right)$ to that of $\SLZ$, we provide full details of the proof of the following lemma.
\begin{lemma}
\label{lem:action_on_gamma_0}
If $\fqmQ$ has level $l$ and $\mat{A}=\left(\begin{smallmatrix}a & b\\
c & d\end{smallmatrix}\right)\in\Gamma_{0}\left(l\right)$ then 
\[
\r\left(\mat{A}\right)\vec{e}_{\alpha}=e\left(bdQ\left(\alpha\right)\right)\ch\left(\mat{A}\right)\vec{e}_{d\alpha}.
\]
\end{lemma}
\begin{proof}
Suppose that $c>0$. Let $n\equiv-b$ (mod $l$) be such that $cn+d>0$.
Then $na\equiv-bda\equiv-b$ (mod $l$) and 
\[
\left(\begin{smallmatrix}a & b\\
c & d\end{smallmatrix}\right)=\left(\begin{smallmatrix}a & an+b\\
c & cn+d\end{smallmatrix}\right)\left(\begin{smallmatrix}1 & -n\\
0 & 1\end{smallmatrix}\right)=\mat{X}\mat{T}^{-n}
\]
with $\mat{X}\in\Gamma_{0}^{0}\left(l\right)$. By Lemmas \ref{sigma_Qd_eps_Qd} and \ref{lem:action_on_gamma_0^0}
it follows that $\eqd{nc+d}=\eqd{d}$ and 
\[
\r(\mat{X})\vec{e}_{\alpha}
=\ch(\mat{X})\vec{e}_{(nc+d)\alpha}
=\eq\left(\mat{X}\right)\eqd{nc+d}^{-1}\vec{e}_{(nc+d)\alpha}
%=\eq\left(X\right)\eqd{d}^{-1}\vec{e}_{nc\alpha+d\alpha}
=\eq\left(\mat{X}\right)\eqd{d}^{-1}\vec{e}_{d\alpha},
\]
where we used that $l|c$ in the last equality. It is therefore enough
to show that $\eq(\mat{X})=\eq(\mat{A})$. If the signature
of $\fqmQ$ is even then $\eq(\mat{X})=\eq(\mat{A})=1$
and we are done. Suppose that the signature is odd. Then $l$ is divisible
by $4$ and we must first show that $\eq(\mat{X})=(\frac{c}{nc+d})$
and $\eq\left(\mat{A}\right)=\left(\frac{c}{d}\right)$ are equal.  
Write $c=2^{k}c_{2}$ where $c_{2}$ is odd and $k\ge2$.
If $8|c$ then $cn+d\equiv d$ (mod $8$) and if $k=2$ then $2^{k}$ is
a square. It follows that $(\frac{cn+d}{c})=(\frac{d}{c})$. 
From the standard quadratic reciprocity law for odd integers \cite[Ch.~I]{cohn:advanced_nt} and the definition of Kronecker's extension of the Jacobi symbol it is easy to show that 
if $x$ and $y$ are any non-zero integers then
\begin{equation}
\Big(\frac{x}{y}\Big)=(x,y)_{\infty}\Big(\frac{y}{x}\Big)e_{8}\big((x_2-1)(y_2-1)\big),\label{eq:quad-recip} 
\end{equation}
where $x_2$ and $y_2$ denote the odd parts of $x$ and $y$.
Using \eqref{eq:quad-recip} twice we see that 
\[
\left(\frac{c}{nc+d}\right)=\left(\frac{nc+d}{c}\right)e_{8}\big((c_{2}-1)(nc+d-1)\big)=\left(\frac{c}{d}\right)e_{8}\big((c_{2}-1)(nc+2d-2)\big),
\]
which is equal to $(\frac{c}{d})$ since $c_{2}$ and $d$ are both odd and $c$ is divisible by $4$.
For $c>0$ we conclude the proof by using  Lemma \ref{lem:sigma(A,BTm)} to show that $\sigma\left(\mat{X},\mat{T}^{-n}\right)=1$ 
and hence that $\r(\mat{A})=\r(\mat{X})\r(\mat{T}^{-n})=\r(\mat{X})\r(\mat{T}^{bd})$.
For $c<0$ we use Lemma \ref{lem:rho(minusA)} together with the fact
that $(\frac{-c}{-d})(\frac{c}{d})=\sign(c)e_{4}(1+d)$.
The case $c=0$ follows immediately from the definition together with Lemma \ref{lem:rho(minusA)} for the case $d=-1$.
\end{proof}
To prove the general formula for $\SLZ$ we essentially repeat the
previous step in going from $\Gamma_{0}^{0}(l)$ to $\Gamma_{0}(l)$,
the main problem is that we need to use elements of
the form $\mat{ST}^{m}\mat{ST}^{n}$, instead of simply $\mat{T}^n$, and therefore have to take care of two parameters
instead of one. 
\subsection{Statement of the main result}
We are now able to give the precise formulation of the main theorem.  
\begin{theorem}\label{thm:main_thm}Let $\fqmQ$
be an FQM with $\fqmQ=(D,Q)$ and bilinear form $B$. Let $\rho_{\fqmQ}$ be the associated
Weil representation,  $\mat{A}=\left(\begin{smallmatrix}a & b\\
c & d\end{smallmatrix}\right)\in\SLZ$ and $\beta\in D$. 
Then 
\[
\r\left(\mat{A}\right)\vec{e}_{\beta}=\xi\left(\mat{A}\right)\sqrt{\frac{\left|D_{c}\right|}{\left|D\right|}}
\sum_{\alpha=x_{c}+c\alpha'}e\left(aQ_c(\alpha)+bdQ\left(\beta\right)+bB\left(\beta,\alpha\right)\right)\vec{e}_{\alpha+d\beta},
 \]
where $Q_c(\alpha)=cQ\left(\alpha'\right)+B\left(\alpha',x_{c}\right)$, $D_{c}$ is the set of elements in $D$ with orders dividing
$c$ and $x_{c}$ is given by Lemma \ref{lem:x_c_def}. The constant
$\xi\left(\mat{A}\right)=\xi\left(a,c\right)$ is an eight root of unity,
given by a fixed Jordan decomposition of $\fqmQ$ as follows:
\begin{itemize}
\item If $c<0$ then $\xi\left(a,c\right)=\xi\left(-a,-c\right)e_{4}\left(\sign\left(\fqmQ\right)\right)$.
\item If $c=0$ then $\xi\left(a,c\right)=1$ if $d>0$, and $\xi\left(a,c\right)=e_{4}\left(\sign\left(\fqmQ\right)\right)$
if $d<0$. 
\item If $c>0$ and $ad=0$ then $\xi\left(a,c\right)=e_{8}\left(-\sign\left(\fqmQ\right)\right)$. 
\end{itemize}
If $c>0$ and $ad\ne0$ then $\xi\left(a,c\right)$ is given by 
\[
\xi\left(a,c\right)=\xi_{0}\prod_{p|\left|D\right|}\xi_{p},
\]
where $\xi_{0}=e_{4}\left(-\sign\left(\fqmQ\right)\right)$
if $\sign\left(\fqmQ\right)$ is even, and if $\sign\left(\fqmQ\right)$ is odd then
\[
\xi_{0}=e_{4}\left(-\sign\left(\fqmQ\right)\right)\left(\frac{-a}{c}\right)\times
\begin{cases}
1 & \text{if $c$ is odd},\\
e_{8}\left(\left(c_{2}+1\right)\left(a+1\right)\right) & \text{if $c$ is even}.
\end{cases}
\]
Furthermore, for $p\ne2$ 
\[
\xi_{p}=\prod_{2<p|q}\left(\frac{-a}{q_{c}^{n_{q}}}\right)\prod_{2<p|q\nmid c}\gammafak_{p}\big(\left(q/q_{c}\right)^{\epsilon_{q}n_{q}}\big)
\left(\frac{c/q_{c}}{\left(q/q_{c}\right)^{n_{q}}}\right).
\]
If $c$ is odd then 
\[
\xi_{2}=e_{8}\left(c\,\oddity\left(\fqmQ\right)\right)\prod_{2|q}\left(\frac{c}{q^{n_{q}}}\right),
\]
and if $c$ is even then 
\begin{align*}
\xi_2 &=e_{8}\left(-\left(1+a\right)\oddity\left(\fqmQ\right)\right) 
\prod_{2|q}\left(\frac{-a}{q_{c}^{n_{q}}}\right)\\
&\times \prod_{2|q\nmid c} e_{8}\left(\frac{-ac}{q_{c}}\oddity\left(\left(\frac{q}{q_{c}}\right)_{*}^{\epsilon_{q}n_{q}}\right)\right)\left(\frac{c/q_{c}}{\left(q/q_{c}\right)^{n_{q}}}\right).
\end{align*}

Here the products are over all non-trivial Jordan components of $\fqmQ$
of order $q$ with $q$ a power of $2$ or a prime $p>2$, $q_{c}=\left(q,c\right)$
and $c_{2}$ is the odd part of $c$. The factors $\gammafak_{2}\left(\cdot\right)$
and $\gammafak_{p}\left(\cdot\right)$ are defined by \eqref{eq:def:gamma_p},
and $*=t_{q}$ if the $2$-adic component of order $q$ is odd.
\end{theorem}

\begin{proof}
We prove the theorem in a sequence of steps, similar to those
in the proof by Scheithauer \cite[Thm.\ 4.7]{scheithauer:weil_rep}. The 
details are presented in the next section.
\end{proof}
\begin{remark}
It is not difficult to verify that when the signature is even then the formula in Theorem \ref{thm:main_thm} reduces
to the one by Scheithauer \cite[Thm.\ 4.7]{scheithauer:weil_rep} for the dual representation. 
For this purpose it is helpful to recall that $\r$ is unitary and if the signature is even then the dual representation $\r^{*}$ is given by 
$\r^{*}(\mat{A})=\bar{\rho}_{\fqmQ}(\mat{A})=\r(\mat{A}^{-1})^{\text{t}}$.
\end{remark}

\begin{remark}
\label{rem:main_thm_mat_elt}If $\alpha,\beta\in D$ and $\mat{A}=\left(\begin{smallmatrix}a & b\\
c & d\end{smallmatrix}\right)\in\SLZ$ then the matrix coefficient of $\r$ with index $\alpha,\beta$ is
given by\[
\r\left(\mat{A}\right)_{\alpha\beta}=\xi\left(\mat{A}\right)\sqrt{\left|D_{c}\right|/\left|D\right|}e\left(acQ\left(\alpha'\right)+aB\left(x_{c},\alpha'\right)-bdQ\left(\beta\right)+bB\left(\beta,\alpha\right)\right)\]
if there is an element $\alpha'\in D$ such that $\alpha=d\beta+x_{c}+c\alpha'$
and otherwise $\r\left(\mat{A}\right)_{\alpha,\beta}=0$. 
In the first case the value is independent of the
choice of $\alpha'$. Note that if $c=0$ then $D_{c}=D,$ $x_{c}=0,$
$D^{c*}=\left\{ 0\right\} $ and thus $\r\left(\mat{A}\right)_{\alpha\beta}=\delta_{\alpha,d\beta}\xi\left(\mat{A}\right)e\left(bdQ\left(\beta\right)\right).$
\end{remark}

In general it does not seem to be possible to simplify the expression
of the factor $\xi\left(\mat{A}\right)$. However, for the discriminant
form of Example \ref{exa:Nx^2}, which is connected to classical Jacobi
forms, we have the following corollary. Using elementary properties of Kronecker symbols it is easy to show
 that the formula in this corollary is equivalent to that of \cite[Thm.~3]{lfg}. 
\begin{corollary}
\label{cor:explicit_Nx^2}Let $N\in\ZZ^{+}$ and consider the FQM given by the
discriminant form $\fqmQ_{N}=\left(D,Q\right)$, with $D=\ZZ/2N\ZZ$
and $Q\left(x\right)=\frac{x^{2}}{4N}$ (mod $1$). If $c$ is an integer
we define $z_{c}\in D$ by $z_{c}=N$ if $\left|2N\right|_{2}=\left|c\right|_{2}$
and $z_{c}=0$ otherwise. Let $\mat{A}=\left(\begin{smallmatrix}a & b\\
c & d\end{smallmatrix}\right)\in\SLZ$ and $x,y\in D$. If there exists an integer solution $v$ to $x=dy+z_{c}+cv$ (mod $2N$)
then
\[
\rho_{\fqmQ_{N}}\left(\mat{A}\right)_{x,y}=\xi\sqrt{\left(2N,c\right)/2N}\, e_{4N}\left(acv^{2}+2z_{c}\left(av+by\right)+bdy{}^{2}+2bcvy\right)
\]
and otherwise $\rho_{\fqmQ_{N}}\left(\mat{A}\right)_{x,y}=0$. In
the first case, the expression is independent of the choice of $v$,
and $\xi$ is an eight-root of unity given as follows: If $c\ne0$
write $2N=2^{m}N_{2}$ and $c=2^{n}c_{2}$ with $N_{2}$ and $c_{2}$
odd. Then $\xi=\xi(a,c)$, defined by 
\[
\xi(a,c)=\left(a,c\right)_{\infty}\left(\frac{a2N/\left(2N,c\right)}{c/\left(2N,c\right)}\right)e_{8}\big(c_{2}N_{2}\delta a -c_{2}(N_{2},c_{2}) \big)
\]
where $\delta=1$ if $\left|2N/c\right|_{2}\ge1$ and otherwise $\delta=0$.
For $c=0$ we have $\xi=1$ if $d=1$ and $\xi=i^{-1}$ if $d=-1$. \end{corollary}
\begin{proof}
From Remark \ref{rem:main_thm_mat_elt} it is clear that we need only
evaluate $\xi$. If $c=0$ the formula follows directly from Lemma \ref{lem:action_on_gamma_0}. 
Assume that $c>0$. We need to compare the $p$-adic components,
$\xi_{p}$, $p>2$, for the FQM $\fqmQ_{N'}$ where $N'=2N/\left(c,2N\right)$
with the corresponding components for $\fqmQ_{N}$. The only
difference is in the associated signs $\epsilon_{q}'$ and $\epsilon_{q}$
and we get that
\[
\gammafak_{p}\left(\left(q/q_{c}\right)^{\epsilon_{q}'}\right)=\gammafak_{p}\left(\left(q/q_{c}\right)^{\epsilon_{q}}\right)\left(\frac{\left(2N,c\right)/q_{c}}{q/q_{c}}\right)\]
from which it follows that \[
\prod_{p>2}\xi_{p}=\prod_{2<p|q}\left(\frac{c/q_{c}}{q/q_{c}}\right)\gammafak_{p}\left(\left(q/q_{c}\right)^{\epsilon_{q}}\right)=\left(\frac{c/\left(2N,c\right)}{N_{2}/\left(N_{2},c_{2}\right)}\right)\prod_{2<p|q}\gammafak_{p}\left(\left(q/q_{c}\right)^{\epsilon_{q}'}\right).\]
By applying Milgram's formula for the FQM $\fqmQ_{N'}$
to the right hand side we see that 
\begin{align*}
\prod_{p>2}\xi_{p}
&=\left(\frac{c/\left(2N,c\right)}{N_{2}/\left(N_{2},c_{2}\right)}\right)e_{8}\left(1\right)\gammafak_{2}\left(\left(q/q_{c}\right)_{t'}^{\epsilon'_{q}}\right)^{-1}\\
&=\left(\frac{c/\left(2N,c\right)}{N_{2}/\left(N_{2},c_{2}\right)}\right)e_{8}\left(1-\frac{N_{2}}{\left(N_{2},c_{2}\right)}\right)\left(\frac{N_{2}/\left(N_{2},c\right)}{2^{m}/\left(2^{m},2^{n}\right)}\right).
\end{align*}
Recall that $\oddity\left(\fqmQ_{N}\right)=N_{2}$. 
Assume now that $c$ is even. Then 
\[
\xi_{2}=e_{8}\left(-\left(1+a\right)N_{2}-\delta'ac_{2}\,\oddity\left(q/q_{c}\right)_{t}^{\epsilon_{q}}\right)\left(\frac{-a}{\left(2^{m},2^{n}\right)}\right)\left(\frac{c/\left(2^{m},2^{n}\right)}{2^{m}/\left(2^{m},2^{n}\right)}\right),
\]
where $\delta'=1$ if $2^{m}\nmid2^{n}$ and $0$ otherwise. Using
$\left(a+1\right)\left(c_{2}+1\right)\left(1-N_{2}\right)\equiv0$ (mod $8$)
we can show that  $\xi=\xi_{0}\xi_{2}\prod\xi_{p}$ is equal to 
\[
\left(\frac{-a}{c/\left(2N,c\right)}\right)\left(\frac{c/\left(2N,c\right)}{2N/\left(2N,c\right)}\right)e_{8}\left(-1-\frac{N_{2}}{\left(N_{2},c\right)}+c_{2}N_{2}\left(1+a-a\delta'\right)\right).
\]
By quadratic reciprocity \eqref{eq:quad-recip} and the elementary fact that the square
of any odd integer is congruent to $1$ modulo $8$ we can also show
that  \[
\left(\frac{c/\left(2N,c\right)}{2N/\left(2N,c\right)}\right)\left(\frac{2N/\left(2N,c\right)}{c/\left(2N,c\right)}\right)=e_{8}\left(-c_{2}N_{2}-1+N_{2}\left(N_{2},c_{2}\right)+c_{2}\left(N_{2},c_{2}\right)\right).\]
Combining this with the formula $\left(\frac{-1}{d}\right)=e_{4}\left(1-d\right)$
for odd $d$ we arrive at 
\[
\xi=\left(\frac{a2N/\left(2N,c\right)}{c/\left(2N,c\right)}\right)e_{8}\left(c_{2}N_{2}\delta a-c_{2}\left(N_{2},c_{2}\right)\right)
\]
with $\delta=1-\delta'$. It is easy to verify that the same formula 
holds for odd $c$, in which case $\delta=0$. For $c<0$
the result follows directly by Lemma \ref{lem:rho(minusA)}, which
says that $\r\left(-\mat{A}\right)\vec{e}_{\gamma}=i\r\left(\mat{A}\right)\vec{e}_{-\gamma}$,
and by noting that $\xi\left(-a,-c\right)=\sign\left(a\right)i^{-1}\xi\left(a,c\right).$
\end{proof}

\section{Proof of the main theorem }
\label{sec:Proof-of-Theorem}
The following lemmas are analogues of the corresponding lemmas and
proposition by Scheithauer \cite[Sect.\ 4]{scheithauer:weil_rep}. We use the
notation of the main theorem, that is, let $\fqmQ$
be a finite quadratic module with abelian group $D$, quadratic form $Q$, bilinear form $B$ and level $l$.
Furthermore, let $\rho_{\fqmQ}$  be the Weil representation associated to $\fqmQ$ and let $\mat{A}=\left(\begin{smallmatrix}a & b\\
c & d\end{smallmatrix}\right)\in\SLZ$. As the first step we give a formula for $\r$ on a set of coset-representatives
of $\Gamma_{0}\left(l\right)$ in $\SLZ$. 
Note that if $c=0$ then the theorem is a consequence of Lemma \ref{lem:action_on_gamma_0}. It is therefore enough to consider $c\ne0$.
\begin{lemma}
\label{lem:rho(STmSTn)}Let $m$ and $n$ be integers and suppose
that $m$ is positive. Then 
\[
\r\left(\mat{ST}^{m}\mat{ST}^{n}\right)\vec{e}_{\beta}=\sqrt{{|D_{m}|}/{|D|}}\si^{2}\,\Scheps{m}\,
e\left(nQ\left(\beta\right)\right)\sum_{\gamma\in D^{m*}}e\left(-Q_{m}\left(\gamma\right)\right)\vec{e}_{\gamma-\beta}
\]
where $D_{m}$, $D^{m*}$ and $Q_{m}$ are defined in Section \ref{sec:On-Jordan-Decompositions}, and
$\Scheps{m}$ is defined in Lemma \ref{lem:Scheithauer3.9}.
\end{lemma}
By Lemmas \ref{lem:sigma(A,BTm)} and \ref{lem:sigma(STmSTn)} we
see that $\sigma\left(\mat{S},\mat{T}^{n}\right)=\sigma\left(\mat{ST}^{m},\mat{ST}^{n}\right)=1$
if $m>0$. Hence
\begin{align*}
\r\left(\mat{ST}^{n}\right) & =\r(\mat{S})\r(\mat{T}^{n})\sigma(\mat{S},\mat{T}^{n})=\r(\mat{S})\rho_{\fqmQ}(\mat{T}^{n})\quad\textrm{and}\\
\r\left(\mat{ST}^{m}\mat{ST}^{n}\right) & =\r(\mat{ST}^{m})\r(\mat{ST}^{n}).
\end{align*}
For $\beta\in D$ we have 
\begin{align*}
\r(\mat{ST}^{n})\vec{e}_{\beta} & =\r(\mat{S})\r(\mat{T}^{n})\vec{e}_{\beta}=\frac{\si}{\sqrt{\left|D\right|}}e\big(nQ(\beta)\big)
\sum_{\gamma\in D}e\left(-B\left(\gamma,\beta\right)\right)\vec{e}_{\gamma}\quad\text{and}\\
%\intertext{and}
\r(\mat{ST}^{m}\mat{ST}^{n})\vec{e}_{\beta} & =\frac{\si^{2}}{\left|D\right|}e\big(nQ(\beta)\big)
\sum_{\alpha\in D}\sum_{\gamma\in D}e\big(mQ(\gamma)-B(\gamma,\alpha+\beta)\big)\vec{e}_{\alpha}.
\end{align*}
The inner sum is evaluated with the help of Lemma \ref{lem:Scheithauer3.9} and we get that 
\[
\frac{1}{\sqrt{\left|D\right|}}\sum_{\gamma\in D}e\big(mQ(\gamma)-B(\alpha+\beta,\gamma)\big)
=\sqrt{\left|D_{m}\right|}\,\Scheps{m}\,e\big(-Q_{m}(\alpha+\beta)\big)
\]
if $\alpha+\beta\in D^{m*}$ and otherwise the left hand side is
equal to zero. Hence 
\begin{align*}
\r\left(ST^{m}ST^{n}\right)\vec{e}_{\beta} & =\sqrt{{|D_{m}|}/{|D|}}\si^{2}\,\Scheps{m}\,e\big(nQ(\beta)\big)\sum_{\gamma\in D^{m*}}e\big(-Q_{m}(\gamma)\big)\vec{e}_{\gamma-\beta}.
\end{align*}
For the rest of the section we use the following additional notation
and assumptions.

\begin{assumption}\label{def:main_definition_abcd_mn}Assume that
$m$ and $n$ are integers, with $m$ positive, satisfying the following conditions: 
\begin{align*}
cn-d & >0,\\
\left(cn-d,l\right) & =1,\\
\left(cn-d\right)m & \equiv c\text{ (mod $l$)},\\
cn-d-1\equiv m-c & \equiv0\text{ (mod $8$)}\,\,\text{ if }\,2\nmid c\quad\text{and}\\
an-b\equiv m+ac & \equiv0\text{ (mod $8$)}\,\,\text{ if }\,2|c.
\end{align*}
Furthermore, we set $d'=cn-d,$ $c'=md'-c$, $b'=an-b$, $a'=mb'-a$
and write $c=2^{k}c_{2}$ where $c_{2}$ is odd. 

\end{assumption}The following elementary fact about linear congruence
equations is useful to keep in mind for the proof of the
next lemma: If $r,s,t\in\ZZ$ then the equation 
\[
sx\equiv r\text{ (mod $t$)}
\]
has integer solutions $x\equiv x_{0}$ (mod $\frac{t}{\left(s,t\right)}$)
if and only if $(s,t)|r$.

\begin{lemma}
It is possible to choose integers $m$ and $n$ satisfying the assumption
above. \end{lemma}
\begin{proof}
\textbf{Case 1, $c$ odd: }Since $\left(c,8\right)=1$ it is clear
that $cn-d\equiv1$ (mod $8$) has solutions, and since $\left(c,d\right)=1$
the arithmetic progression $cn-d$ contains an infinite number of
primes not dividing $l$. It follows that there exists $n$ such that
$cn-d>0,$ $\left(cn-d,l\right)=1$ and $d'=cn-d\equiv1$ (mod $8$). Let
$m\equiv c$ (mod $8$) then we can write $m$ as $m=c+8j$ and it is clear
that $d'm\equiv c$ (mod $l$) is equivalent to $d'\left(c+8j\right)\equiv c$ (mod $l$)
$\Leftrightarrow$ $8d'j\equiv c-d'c$ (mod $l$) and since $\left(8d',l\right)=\left(8,l\right)$
and $c-d'c\equiv0$ (mod $8$) it is clear that there exists $m\equiv c$ (mod $8$)
such that $m>0$ and $d'm\equiv c$ (mod $l$).

\textbf{Case 2, $c$ even:} If $n\equiv ab$ (mod $8$) then $an\equiv a^{2}b\equiv b$ (mod $8$)
and $cn-d\equiv abc-d\equiv a^{2}d-a-d\equiv-a$ (mod $8$). 
The set of integers $cn-d\equiv-a$ (mod $8$) contains an infinite number
of primes which does not divide $l$ since $(a,8c)=1$. Hence we can choose $n$ with
$cn-d>0,$ $an-b\equiv0$ (mod $8$) and $\left(cn-d,l\right)=1$. 

Let $m\equiv-ac$ (mod $8$). Then $m=-ac+8j$ for some $j\in\ZZ$
and $d'm\equiv c$ (mod $l$) $\Leftrightarrow$ $\left(-ac+8j\right)d'\equiv c$ (mod $l$)
$\Leftrightarrow$ $8jd'\equiv c\left(1+ad\right)'$ (mod $l$). Since
$\left(8d',l\right)=\left(8,l\right)$ and $c\left(1+ad'\right)\equiv c\left(1-a^{2}\right)\equiv0$ (mod $8$)
it follows that $(8,l)$ divides $c\left(1+ad'\right)b$. Therefore we can
 find an $m>0$ such that $m\equiv-ac$ (mod $8$) and $md'\equiv c$ (mod $l$). 
\end{proof}
From now on suppose that $m$ and $n$ are chosen as above. Define $\mat{X}\in\Gamma_0(l)$ by 
\[
\mat{X}=\mat{A}\mat{T}^{-n}\mat{S}\mat{T}^{-m}\mat{S}=\left(\begin{array}{cc}
a' & b'\\
c' & d'\end{array}\right)=\left(\begin{array}{cc}
\left(an-b\right)m-a & an-b\\
\left(cn-d\right)m-c & cn-d\end{array}\right).
\]
%and observe that $\mat{A}=\mat{XST}^{m}\mat{ST}^{n}$. 
%We proceed step by step in proving the theorem by writing 
We first write $\rho_{\fqmQ}(\mat{A})$ in terms of $\rho_{\fqmQ}(\mat{X})$ and $\rho_{\fqmQ}(\mat{ST}^{m}\mat{ST}^{n})$, 
and then we evaluate the resulting expressions and show that they are independent of the choice of $m$ and $n$.
Since $m>0$ it follows from \eqref{eq:co-cycle-for-rho} and Lemma \ref{lem:sigma(xStmSTn)=00003D1} that
 \[
\r(\mat{A})=\sigma\left(\mat{X},\mat{ST}^{m}\mat{ST}^{n}\right)\r(\mat{X})\r(\mat{ST}^{m}\mat{ST}^{n})=\r(\mat{X}\r(\mat{ST}^{m}\mat{ST}^{n}).
\]
By Lemma \ref{lem:action_on_gamma_0} and \ref{lem:rho(STmSTn)} we see that if $\beta\in D$ then
\begin{align*}
\r(\mat{A})\vec{e}_{\beta} & =\r(\mat{X})\rho(\mat{ST}^{m}\mat{ST}^{n})\vec{e}_{\beta}\\
 & =\r(\mat{X})\sqrt{|D_m|/|D|}\si^{2}\,\Scheps{m}\,e\big(nQ(\beta)\big)\sum_{\gamma\in D^{m*}}e\big(-Q_{m}(\gamma)\big)\vec{e}_{\gamma-\beta}\\
 & =\sqrt{|D_m|/|D|}\si^{2}\,\Scheps{m}\,\sum_{\gamma\in D^{m*}}e\big(nQ(\beta)-Q_{m}(\gamma)\big)\rho_{\fqmQ}(\mat{X})\vec{e}_{\gamma-\beta} \\
 & =\Lambda\cdot\sum_{\gamma\in D^{m*}}e\big(nQ(\beta)-Q_{m}(\gamma)+b'd'Q(\gamma-\beta)\big)
\vec{e}_{d'\left(\gamma-\beta\right)},
\end{align*}
where $\Lambda=\sqrt{|D_m|/|D|} \si^{2}\,\Scheps{m}\,\eq(\mat{X})\eqd{d'}^{-1}$.
By the arguments of \cite[pp.~16-17]{scheithauer:weil_rep}  we see that
 \[ 
nQ\left(\beta\right)-Q_{m}\left(\gamma\right)+b'd'Q\left(\gamma-\beta\right)\equiv aQ_{c}\left(\mu\right)+bdQ\left(\beta\right)+bB\left(\beta,\mu\right) \text{ (mod $1$)},
\]
where $\mu=d'\left(\gamma-\beta\right)-d\beta\in D^{c*}$. Furthermore,
as $\gamma$ runs through $D^{m*}=D^{c*}$ so does $\mu$, and we can therefore rewrite the last sum as 
\[
%\sum_{\gamma\in D^{m*}}e\big(nQ(\beta)-Q_{m}(\gamma)+b'd'Q(\gamma-\beta)\big)\vec{e}_{d'(\gamma-\beta)}=
\sum_{\mu\in D^{c*}}e\big(aQ_{c}(\mu)+bdQ(\beta)+bB(\beta,\mu)\big)\vec{e}_{\mu+d\beta}.
\]
It follows that if $\alpha-d\beta\notin D^{c*}$ then  $\r(M)_{\alpha\beta}=0$, 
and otherwise 
\begin{align*}
\r(M)_{\alpha,\beta} & =\xi\lc e\big(aQ_{c}(\alpha-d\beta)+bdQ(\beta)+bB(\beta,(\alpha-d\beta))\big)\\
 & =\xi\lc e\big(aQ_{c}(\alpha-d\beta)-bdQ(\beta)+bB(\beta,\alpha)\big),
\end{align*}
where $\xi$ is the eight root of unity given by
\[
\xi=\si^{2}\,\Scheps{m}\,\eq\left(X\right)\eqd{d'}.
\]
By Lemma \ref{sigma_Qd_eps_Qd} we see that $\eqd{d'}=(\frac{d'}{\left|D\right|})e_{8}\big(-(1-d')\oddity(\fqmQ)\big)$,
and by Lemma \ref{lem:Scheithauer3.9} (using that $q_{m}=(q,m)=(q,c)=q_{c}$)
we know that $\Scheps{m}$ is given by  
\[
\prod_{2|q\nmid c}e_{8}\left(\frac{m}{q_{c}}\oddity\left(\frac{q}{q_{c}}\right)_{*}^{\epsilon_{q}n_{q}}\right)\left(\frac{m/q_{c}}{\left(q/q_{c}\right)^{n_{q}}}\right)
\prod_{2<p|q\nmid c}\gammafak_{p}\big((q/q_{c})^{\epsilon_{q}n_{q}}\big)\left(\frac{m/q_{c}}{\left(q/q_{c}\right)^{n_{q}}}\right).
\]
We are now able to write $\xi=\xi_{0}\prod\xi_{p}$ with
\begin{align*}\label{eq:xi_2_can_be}
\xi_{0} & =\si^{2}\eq(\mat{X}),\\
\xi_{2} & =e_{8}\big((d'-1)\oddity(\fqmQ)\big) \prod_{2|q}\left(\frac{d'}{q^{n_{q}}}\right)
 \prod_{2|q \nmid c}e_{8} \left(\frac{m}{q_{c}}\,\oddity\left(\frac{q}{q_{c}}\right)_{*}^{\epsilon_{q}n_{q}}\right)
 \left(\frac{m/q_{c}}{\left(q/q_{c}\right)^{n_{q}}}\right),\\
\xi_{p} & =\prod_{2<p|q\nmid c}\gammafak_{p}\left(\left(q/q_{c}\right)^{\epsilon_{q}n_{q}}\right)\left(\frac{m/q_{c}}{\left(q/q_{c}\right)^{n_{q}}}\right)\prod_{2<p|q}\left(\frac{d'}{q^{n_{q}}}\right).
\end{align*}
To conclude the proof of the theorem for the case $c>0$, we use
Lemmas \ref{lem:kronecker(c',d')} (for $\xi_{0}$), \ref{lem:xi2}
(for $\xi_{2}$) and \ref{lem:xip} (for $\xi_{p}$) to remove the
dependence on $m$ and $n$.
\begin{lemma}
\label{lem:kronecker(c',d')}If the signature of $\fqmQ$ is
even then $\xi_{0}=\si^{2}$, and if it is odd then 
\[
\xi_{0}=\si^{2}\left(\frac{-a}{c}\right)\times
\begin{cases}
1 & \text{if $c$ is odd,}\\
e_{8}\big((c_{2}+1)(a+1)\big) & \text{if $c$ is even.}
\end{cases}
\]
\end{lemma}
\begin{proof}
The case of even signature is trivial, so assume that the signature
is odd. By Assumption \ref{def:main_definition_abcd_mn} we have $c>0,$ $cn-d>0$, $cn-d$ is odd (since $4|l$ when $\sign\left(\fqmQ\right)$
is odd), $cn-d\equiv1$ (mod $8$) if $c$ is odd and $cn-d\equiv-a$ (mod $8$)
if $c$ is even. By \eqref{eq:quad-recip} and periodicity of the Legendre symbol we see that 
\begin{align*}
\eq(\mat{X})&=\left(\frac{c'}{d'}\right)= \left(\frac{\left(cn-d\right)m-c}{cn-d}\right)=\left(\frac{-c}{cn-d}\right)=\left(\frac{-c_{2}}{cn-d}\right)\left(\frac{2^{k}}{cn-d}\right)\\
 & =e_{8}\big((cn-d-1)(-c_{2}-1)\big)\left(\frac{cn-d}{-c_{2}}\right)\left(\frac{cn-d}{2^{k}}\right)\\
 & =\left(\frac{-d}{c_{2}}\right)\left(\frac{-a}{2^{k}}\right)
\begin{cases}
1 & \text{if $c$ is odd},\\
e_{8}\left((a+1)(c_{2}+1)\right) & \text{if $c$ is even}.
\end{cases}
\end{align*}
We finish the proof by observing that 
%The conclusion follows from the fact that 
$(\frac{d}{p})=(\frac{a}{p})$
for all odd primes $p$ dividing $c$. 
\end{proof}
\begin{lemma}
\label{lem:xi2}The factor $\xi_2$ is given by 
\[
\xi_{2}=\begin{cases}
e_{8}\big(c\,\oddity\left(\fqmQ\right)\big)\prod_{2|q}\left(\frac{c}{q^{n_{q}}}\right) & \text{if $c$ is odd},\\
e_{8}\big(-\left(1+a\right)\oddity\left(\fqmQ\right)\big) \prod_{2|q}\left(\frac{-a}{q_{c}^{n_{q}}}\right) \times\\
\times \prod_{2|q\nmid c}e_{8}\left(\frac{-ac}{q_{c}}\,\oddity\left(\frac{q}{q_{c}}\right)_{*}^{\epsilon_{q}n_{q}}\right)\left(\frac{c/q_{c}}{\left(q/q_{c}\right)^{n_{q}}}\right)
 & \text{if $c$ is even.}
\end{cases}
\]
\end{lemma}
\begin{proof}
It is clear that $\xi_{2}=1$ if $\left|D\right|$ is odd; and we therefore assume, without loss of generality, that $\left|D\right|$ is
even. We have seen that 
\[
\xi_{2} = e_{8}\big((d'-1)\oddity(\fqmQ)\big)\prod_{2|q}\left(\frac{d'}{q^{n_{q}}}\right)
\prod_{2|q\nmid c}e_{8}\left(\frac{m}{q_{c}}\,\oddity\left(\frac{q}{q_{c}}\right)_{*}^{\epsilon_{q}n_{q}}\right)\left(\frac{m/q_{c}}{\left(q/q_{c}\right)^{n_{q}}}\right),
\]
and to remove the dependence of this expression on $m$ and $n$ we must first distinguish between the case of $c$ being odd and even.
If $c$ is odd then $q_c=1$  and $m$ is odd. Hence
\[
\xi_{2} = e_{8}\big(\left(m-1+d'\right)\oddity\left(\fqmQ\right)\big)\prod_{2|q}\left(\frac{d'm}{q^{n_{q}}}\right).
 %\prod_{2|q}e_{8}\left(\left(m-1+d'\right)\oddity\left(q_{*}^{\epsilon_{q}n_{q}}\right)\right)\left(\frac{d'm}{q^{n_{q}}}\right)\\
\]
From Assumption \ref{def:main_definition_abcd_mn} we know that $d'\equiv1$ (mod $8$)
and $m-1+d'\equiv c$ (mod $8$). It follows that $md'\equiv m\equiv c$ (mod $8$) and
\[
\xi_{2}=e_{8}\big(c\,\oddity(\fqmQ)\big)\prod_{2|q}\left(\frac{c}{q^{n_{q}}}\right).
\]
If $c$ is even then $a$ is odd, $b'\equiv an-b\equiv0$ (mod $8$)
and $m\equiv-ac$ (mod $8$). It follows that $n\equiv ab$ (mod $8$) and since
$a'd'-b'c'=1$ it is clear that $a'd'\equiv1$ (mod $8$), which implies
that $d'\equiv a'=mb'-a\equiv-a$ (mod $8$). Thus $1-d'\equiv-\left(1+a\right)$ (mod $8$).
Finally, by collecting the Kronecker symbols with $-a$ in the denominator, observing that $(q,c)=q$ if $q|c$, we arrive at the desired formula.
\end{proof}
\begin{lemma}
\label{lem:xip}If $p>2$ is a prime dividing $\left|D\right|$ then
the factor $\xi_p$ is given by 
\[
\xi_{p}=\prod_{2<p|q}\left(\frac{-a}{q_{c}^{n_{q}}}\right)\prod_{2<p|q\nmid c}\gammafak_{p}\left(\left(q/q_{c}\right)^{\epsilon_{q}n_{q}}\right)\left(\frac{c/q_{c}}{\left(q/q_{c}\right)^{n_{q}}}\right).
\]
\end{lemma}
\begin{proof}
We know that 
\[
\xi_{p}=\prod_{p|q\nmid c}\gammafak_{p}\left(\left(q/q_{c}\right)^{\epsilon_{q}n_{q}}\right)\left(\frac{m/q_{c}}{\left(q/q_{c}\right)^{n_{q}}}\right)\prod_{p|q}\left(\frac{d'}{q^{n_{q}}}\right),
\]
and since $d'm\equiv c$ (mod $l$) it follows that $m\equiv a'd'm\equiv a'c$ (mod $l$)
and consequently $\frac{m}{q_{c}}\equiv a'\frac{c}{q_{c}}$ (mod $\frac{l}{q_{c}}$)
for all $q$. Furthermore, if $q\nmid c$ then $\frac{m}{q_{c}}\equiv a'\frac{c}{q_{c}}$ (mod $p$)
for all $p|q$. There are two cases to consider. If $p\nmid c$ then
\[
\xi_{p}=\prod_{2<p|q}\gammafak_{p}\left(q^{\epsilon_{q}n_{q}}\right)\left(\frac{md'}{q^{n_{q}}}\right)=\prod_{2<p|q}\gammafak_{p}\left(q^{\epsilon_{q}n_{q}}\right)\left(\frac{c}{q^{n_{q}}}\right).
\]
If $p|c$ then $d'=cn-d\equiv-d$ (mod $p$), $p\nmid d$ and $p|m$. Hence
$a'=b'm-a\equiv-a$ (mod $p$), $m\equiv-ac$ (mod $p$) and  
\begin{align*}
\xi_{p}&=\prod_{2<p|q\nmid c}\gammafak_{p}\left(\left(q/q_{c}\right)^{\epsilon_{q}n_{q}}\right)\left(\frac{-ac/q_{c}}{\left(q/q_{c}\right)^{n_{q}}}\right)\prod_{2<p|q}\left(\frac{-d}{q^{n_{q}}}\right) \\
       &=\prod_{2<p|q}\left(\frac{-a}{q_{c}^{n_{q}}}\right)\prod_{2<p|q\nmid c}\gammafak_{p}\left(\left(q/q_{c}\right)^{\epsilon_{q}n_{q}}\right)\left(\frac{c/q_{c}}{\left(q/q_{c}\right)^{n_{q}}}\right).
\end{align*}
Since $q_{c}=1$ if $p\nmid c$ it follows that this formula
actually contains both cases. 
\end{proof}
We have now proved the theorem for $c>0$. For the case of $c<0$
we observe that $D^{-c*}=D^{c*}$ and $Q_{-c}\left(\alpha\right)=-Q_{c}\left(\alpha\right)$,
and apply the theorem to the matrix $-\mat{A}$. Then 
\begin{align*}
\r(-\mat{A})\vec{e}_{-\beta} & =\xi(-a,-c)\lc\sum_{\alpha\in D^{c*}}e(-aQ_{-c}(\alpha)+bdQ(-\beta)-bB(\alpha,-\beta))\vec{e}_{\alpha+d\beta}\\
 & =\xi(-a,-c)\lc\sum_{\alpha\in D^{c*}}e(aQ_{c}\left(\alpha\right)+bdQ\left(\beta\right)+bB\left(\alpha,\beta\right))\vec{e}_{\alpha+d\beta},
\end{align*}
where $l_c=\sqrt{|D_c|/|D|}$. By Lemma \ref{lem:rho(minusA)} we find that %$\r(\mat{A})\vec{e}_{\beta} =-\si^{2}\r(-\mat{A})\vec{e}_{-\beta}$
\begin{align*}
\r(\mat{A})\vec{e}_{\beta} & =-\si^{2}\r(-\mat{A})\vec{e}_{-\beta}\\
 & =\si^{-2}\xi(-a,-c)\lc\sum_{\alpha\in D^{c*}}e\big(aQ_{c}(\alpha)+bdQ(\beta)+bB(\alpha,\beta)\big)\vec{e}_{\alpha+d\beta},
\end{align*}
and hence $\xi(a,c)=e_{4}\big(\sign(\fqmQ)\big)\xi(-a,-c)$ if $c<0$. 
This concludes the proof of Theorem \ref{thm:main_thm}. 
\begin{acknowledgements}

I would like to thank Nils Scheithauer for clarifying details of \cite{scheithauer:weil_rep},
Nils-Peter Skoruppa for sharing thoughts about Weil representations
in general, and the manuscript \cite{skoruppa-weilrep},
in particular. I would also like to thank Stephan Ehlen for his assistance with
relating the cocycles $\sigma$ and $\mu$ in Section
\ref{sec:The-Metaplectic-Group}. 
\end{acknowledgements}
%\bibliographystyle{plain}
%\bibliography{/home/fredrik/Documents/matematik/refs}
%\input fqm_weil_representation_bib.tex
\def\cprime{$'$}

\end{document}